\documentclass[reqno]{amsart}
\usepackage{amsfonts}
\usepackage{latexsym,amssymb}
\usepackage{amsmath}

\newtheorem{theorem}{Theorem}[section]%[chapter]
\newtheorem{lemma}{Lemma}[section]

\newtheorem{remark}[theorem]{Remark}
\newtheorem{definition}[theorem]{Definition}
\numberwithin{equation}{section}
\newcommand{\secao}[1]{\section{#1}\setcounter{equation}{0}}

\begin{document}
\today
\title[Coupled third-order nonlinear Schr\"{o}dinger System]{On the Periodic Cauchy problem
for a coupled system of third-order nonlinear Schr\"{o}dinger
equations}
\author{M. Scialom*}
%\thanks{This work was partially supported by FAPESP Brazil.}
\address{IMECC-UNICAMP\\
13083-970, Campinas, S\~ao Paulo, Brazil}
\email{scialom@ime.unicamp.br}

\author{L. M. Bragan\c {c}a}
\address{DMA-UFV\\  36570-000, Vi\c{c}osa, Minas Gerais, Brazil.}
\email{mendonca@ufv.br}

\keywords{coupled system of third-order nonlinear Schr\"{o}dinger
equations, periodic Cauchy problem, local \& global well-posedness
\\ *Corresponding author: E-mail scialom@ime.unicamp.br, Tel. +55 19 35216042, Fax +55 19 32895766}
\subjclass[2000]{35Q35, 35Q53}
\begin{abstract}
We investigate some well-posedness issues for the initial value
problem (IVP) associated to the system
\begin{equation*}
\left\{
\begin{array}
[c]{l}%
2i\partial_{t}u+q\partial_{x}^{2}u+i\gamma\partial_{x}^{3}u=F_{1}(u,w)\\
2i\partial_{t}w+q\partial_{x}^{2}w+i\gamma\partial_{x}^{3}w=F_{2}(u,w)\text{,}%
\end{array}
\right.
\end{equation*}
where $F_{1}$ and $F_{2}$ are polynomials of degree 3 involving
$u$, $w$ and their derivatives. This system describes the dynamics
of two nonlinear short-optical pulses envelopes $u(x,t)$ and
$w(x,t)$ in fibers (\cite{31}, \cite{14}). We prove periodic local
well-posedness for the IVP with data in Sobolev spaces
$H^{s}(\mathbb{T)\times} H^{s}(\mathbb{T)}$, $ s\geq 1/2$ and
global well-posedness result in Sobolev spaces
$H^{1}(\mathbb{T)\times }H^{1}(\mathbb{T)}$.
\end{abstract}

\maketitle

\secao{Introduction}

Consider the initial value problem (IVP)
\begin{equation}\label{1.1}
\left\{
\begin{array}
[c]{l}%
2i\partial_{t}u+q\partial_{x}^{2}u+i\gamma\partial_{x}^{3}u=F_{1}(u,w),
\; \;\;x\in\mathbb{T}, \;t>0,\\
2i\partial_{t}w+q\partial_{x}^{2}w+i\gamma\partial_{x}^{3}w=F_{2}(u,w)\text{,}\\
u(x,0)=u_{0},\text{\hskip10pt} w(x,0)=w_{0}\text{,}%
\end{array}
\right.
\end{equation}
where $u$, $w$ are complex valued functions,
$$
F_{1}(u,w)\!=\!-2i\beta(|u|^2+\sigma_{\beta}|w|^2)  \partial_{x}u-2\alpha u(|u| ^{2}%
+\sigma_{\alpha}|w| ^{2})-2i\mu u\partial_{x}(|u|^2+\sigma_{\mu}|w| ^{2}),
$$
$F_{2}(u,w)=F_{1}(w,u)$ and $q$, $\gamma$, $\beta$, $\mu$, $\alpha$, $\sigma_{\alpha}$,
$\sigma_{\beta}$ and $\sigma _{\mu}$ are real parameters.

%%%%%%%%%%%%%%%%%%%%%%%%%%%%%%%%%%%%%%%%%%%%%%%
This system describes the dynamics of two nonlinear short-optical
pulses envelopes $u(x,t)$ and $w(x,t)$ in fibers. This model
is formed by a pair of Schr\"odinger-Korteweg-de Vries equations
coupled through nonlinear terms and it was derived by Porsezian,
Shanmugha Sundaram and Mahalingam \cite{31}. It generalizes
the model (\ref{1.3}) derived by Hasegawa-Kodama \cite{14}.

There is a growing interest in studying the propagation of optical
soliton pulses in fiber. This is because of their potential
applications in fiber-optic-based communication systems. The idea
of soliton based all-optical communication systems with loss
compensated by optical amplifications has provided hints of
potential advantage for solitons over conventional systems. The
major attraction for the soliton communication system arises from
the fact that repeater spacing for this kind of system could be
much larger than that required by the conventional systems.

The system above has been previously studied (\cite{28} and
\cite{39}) in the particular case
$\sigma_{\alpha}=\sigma_{\beta}=\sigma_{\mu}=1$ and the system of
Hirota and Hirota-Satsuma studied by \cite{2} and \cite{1}
respectively. Radhakrishnan and Lakshmanan \cite{33} used the
following transformation of variables in system (\ref{1.1})
\begin{align*}
u(x,t)  &  =u_{1}\left(  x-\frac{q^{2}}{6\gamma}t,t\right)  \exp
i\left(
\frac{q}{3\gamma}x-\frac{q^{3}}{27\gamma^{2}}t\right), \\
w(x,t)  &  =w_{1}\left(  x-\frac{q^{2}}{6\gamma}t,t\right)  \exp
i\left(
\frac{q}{3\gamma}x-\frac{q^{3}}{27\gamma^{2}}t\right)  \text{,}%
\end{align*}
under the particular conditions $\sigma_{\alpha}=\sigma_{\beta}$ and
$q\beta =3\gamma\alpha$, to obtain the following form of coupled
envelope equations
corresponding to the system (\ref{1.1})
\begin{equation}\label{1.2}
\begin{cases}
2\partial_{t}u_1\!+\!\gamma\partial_{x}^{3}u_{1}\!+\! 2\beta\left(|u_1|^{2}\!+
\!\sigma_{\beta}|w_1|^{2}\right)
\partial_{x}u_1\!+\!2\mu u_{1}\partial_{x}(|u_1|^2\!+\!\sigma_{\mu}|w_1|^2)\!=\!0,\\
2\partial_{t}w_{1}\!+\!\gamma\partial_{x}^{3}w_{1}\!+\!2\beta\left(|w_1| ^{2}\!+
\!\sigma_{\beta}|u_1|^2\right)
\partial_{x}w_{1}\!+\!2\mu w_{1}\partial_{x}(|w_1|^2\!+\!\sigma_{\mu}|u_1|^2)\!=\!0.
\end{cases}
\end{equation}

Then, they applied the Hirota bilinear transformation (see \cite{17}) to
(\ref{1.2}) to construct dark and bright soliton solutions to (\ref{1.1})
assuming further that $\beta=\mu$, $q\beta=3\gamma\alpha$, $\gamma\neq0$ and
$\sigma_{\alpha}=\sigma_{\beta}=\sigma_{\mu}=1$. Recently,
Porsezian and Kalithasan \cite{32} discussed the construction of
new cnoidal wave solutions and found exact solutions of both bright
and dark solitary wave to system (\ref{1.1}).

As far as we know, the previous works in this subject do not
address well-posedness issues for the system (\ref{1.1}), so our
aim is to fill up this gap.

Note that if the pulse
 $w_{1}=0$, the system (\ref{1.2}) reduces to the well known
 modified complex KdV equation. This fact suggests that the results obtained
 for the periodic modified
 KdV equation should be the ones we expect for the system (\ref{1.1}).

When $w=0$ the system (\ref{1.1}) reduces to equation
\begin{equation}\label{1.3}
i\partial_{t}u+\frac{q}{2}\partial_{x}^{2}u+i\frac{\gamma}{2}\partial_{x}%
^{3}u+\alpha u| u| ^{2}+i\left(
\beta+\mu\right) | u| ^{2}\partial_{x}u+i\mu
u^{2}\partial_{x}\overline
{u}=0\text{,}
\end{equation}
which describes the dynamics of one single nonlinear pulse in an
optic fiber.

The initial value problem associated to equation  (\ref{1.3}) was
considered by several authors (\cite{26}, \cite{35}, \cite{6},
\cite{38}, \cite{7} and \cite{8}) in $H^{s}(\mathbb{R})$, where $
s\geq 1/4$ is the best result.

In the case of system (\ref{1.1}), Scialom and Bragan\c{c}a
\cite{4a} obtained local well-posedness solution in Sobolev spaces
$H^{s}(\mathbb{R)\times }H^{s}(\mathbb{R)}$, $ s\geq 1/4$, and
global well-posedness in $H^{s}(\mathbb{R)\times
}H^{s}(\mathbb{R)}$,  $s > 3/5$.

In the periodic setting, Takaoka in \cite{38}, considered the IVP
$(1.3)$ and showed local well-posedness in $H^{s}(\mathbb{T)}$,
$s\geq \frac{1}{2}$.

For the IVP associated to system (\ref{1.1}), we are able to
obtain local well-posedness for initial data in
$H^{s}(\mathbb{T)\times} H^{s}(\mathbb{T)}$, $s\geq \frac{1}{2}$
as in the single equation case. The approach we use is similar to
the one given in \cite{38} though the presence of the coupled
terms in (\ref{1.1}) make the estimates more involved.

To describe our local result we need the definitions and
notations.

\begin{definition}
Let $\mathcal{P}=\mathcal{C}^{\infty}\left(  \mathbb{T}\right)
=\left\{ g:\mathbb{R}\rightarrow\mathbb{C};\text{
}g\in\mathcal{C}^{\infty}\text{periodic with period} 2\pi\right\}
$. $\mathcal{P}^{\prime}$(dual of $\mathcal{P}$) is the collection
of all linear functionals  from $\mathcal{P}$ to $\mathbb{C}$.
$\mathcal{P}^{\prime}$ is periodic distributions. If
$g\in\mathcal{P}^{\prime}$ denote the value of  $g$ in $\varphi$
by $g\left( \varphi\right) =\langle g,\varphi\rangle$. Consider
the functions  $\theta_{n}\left( x\right) =e^{inx}$,
$n\in\mathbb{Z}$ and $x\in\mathbb{R}$. The Fourier transform
$g\in\mathcal{P}^{\prime}$ is the function
$\widehat{g}:\mathbb{Z}\rightarrow\mathbb{C}$ defined by
$\widehat{g}\left( n\right)  =\langle g,\theta_{-n}\rangle$. If
$g$ is periodic of period $2\pi$, for instance, $g\in L^{1}\left(
\mathbb{T}\right)  $ then
\[
\widehat{g}(n)=\int_{\mathbb{T}}e^{-ixn}g(x)dx=\frac{1}{2\pi}\int
\nolimits_{0}^{2\pi}e^{-ixn}g(x)dx\text{,}%
\]
where $\mathbb{T=R}/2\pi\mathbb{Z}$ represents the one dimensional
torus.
\end{definition}

\begin{definition}
Let $s\in\mathbb{R}$, the Sobolev space $H^{s}(\mathbb{T)}$ is the
set of all $g\in\mathcal{P}^{\prime}$ such that
\[
\| g\| _{H^{s}(\mathbb{T})}=\left(2\pi\sum\limits_{n\in
\mathbb{Z}}(1+| n| ^{2})^{s}| \widehat{g}(n)| ^{2}\right)
^{\frac{1}{2}}<\infty.
\]
In this work we assume that $g$ is periodic of period $2\pi$.
\end{definition}

\begin{definition}
We denote by $\widetilde{f}$ the Fourier transform of $f$ in
relation to space-time variables
\[
\widetilde{f}(n,\tau)= {\displaystyle\int\nolimits_{\mathbb{R}}}
\int_{\mathbb{T}}e^{-i(xn+t\tau)}f(x,t)dxdt\text{.}%
\]
The inverse Fourier transform is given by
\[
f(x,t)=\sum\limits_{n\in\mathbb{Z}}e^{inx}
{\displaystyle\int\nolimits_{\mathbb{R}}}
e^{it\tau}\widetilde{f}(n,\tau)d\tau\text{.}%
\]
The Fourier transform of $fgh$, where $f=f(x,t)$, $g=g(x,t)$ and
$h=h(x,t)$ are periodic functions with respect to  the $x$
variable obtained as
\[
\widetilde{fgh}(n,\tau)\!=\!(\widetilde{f}\ast\widetilde{g}\ast\widetilde{h})(n,\tau)
\!=\!\sum_{n_{1}\in\mathbb{Z}}\sum_{n_{2}\in\mathbb{Z}}\iint\nolimits_{\mathbb{R}
^{2}}\widetilde{f}(n_{1},\tau_{1})\widetilde{g}(n_{2},\tau_{2})\widetilde
{h}(n-n_{1}-n_{2},\tau-\tau_{1}-\tau_{2})d\tau_{1}d\tau_{2}\text{.}
\]

\end{definition}

\begin{definition}
Let $\mathcal{V}$ be the space of functions $f$ such that
\begin{itemize}
\item[(i)]  $f:\mathbb{T}\times\mathbb{R}\rightarrow\mathbb{C}$,
\item[(ii)]  $f(x,\cdot)\in\mathcal{S}(\mathbb{R})$ for each
$x\in\mathbb{T}$, \item[(iii)]  $f(\cdot,t)\in
C^{\infty}(\mathbb{T)}$ for each $t\in\mathbb{R}$.
\end{itemize}
We define the space $X_{s,b}$ associated to operator
$\partial_{t}-i\frac{q}{2}
\partial_{x}^{2}+\partial_{x}^{3}+c_{0}\partial_{x}$ as the completion of $\mathcal{V}$
with respect to the following norm%
\[
\| f\| _{X_{s,b}}:=\| f\| _{(1,s,b)}=\| \langle
n\rangle^{s}\langle
\tau-n^{3}+\frac{q}{2}n^{2}+c_{0}n\rangle^{b}\widetilde{f}(n,\tau)\|
_{l_{n}^{2}L_{\tau}^{2}}\text{,}
\]
where $\langle n\rangle=(1+|n|^{2})^{\frac{1}{2}}$ and $s$, $b$,
$c_{0}\in\mathbb{R}$.

The space $Z_{s,b}$ is the completion of $\mathcal{V}$ with
respect to the norm
\[
\| f\| _{Z_{s,b}}:=\| f\| _{(2,s,b)}=\| \langle
n\rangle^{s}\langle
\tau-n^{3}+\frac{q}{2}n^{2}+c_{0}n\rangle^{b}\widetilde{f}(n,\tau)\|
_{l_{n}^{2}L_{\tau}^{1}}\text{,}
\]
and we consider the space $Y_{s}=X_{s,\frac12}\cap Z_{s,0}$ with
the norm
\[
\| f\| _{Y_{s}}=\| f\| _{(1,s,%
%TCIMACRO{\U{bd}}%
%BeginExpansion
\frac12
%EndExpansion
)}+\| f\| _{(2,s,0)}\text{.}
\]
\end{definition}

\begin{remark}
Note that for $b>\frac{1}{2}$ we have that  $X_{s,b}\subset C(\mathbb{R}_{t};H^{s}%
(\mathbb{T}))$ and
 for $b=0$, we have $Z_{s,0}\subset$
$C(\mathbb{R}_{t};H^{s}(\mathbb{T}))$ and $Y_{s}\subset C(\mathbb{R}_{t}%
;H^{s}(\mathbb{T}))$.

The space $Y_{s}^{T}=\left\{  f\mid_{\lbrack-T,T]}\,:\,f\in
Y_{s}\right\} $ with the norm
\[
\| f\| _{Y_{s}^{T}}=\inf\left\{  \|g\| _{Y_{s}}\;:\;
g\mid_{_{\lbrack-T,T]}}=f\text{\hskip5pt and\hskip5pt}g\in
Y_{s}\right\},
\]
satisfies $Y_{s}^{T}\subset C([-T,T];H^{s}(\mathbb{T}))$.
\end{remark}

Now, we are in position to state the local result.

\begin{theorem}\label{teo1.7}
Suppose that $\frac{q}{3}\notin\mathbb{Z}$. If $s\geq\frac{1}{2}$
and $\overrightarrow{u_0}=\left(  u_{0},w_{0}\right)  \in H^{s}(\mathbb{T}%
)\times H^{s}(\mathbb{T})$, then there exist $T(\| \overrightarrow{u}%
_{0}\| _{H^{s}})>0$ and a unique solution
$\overrightarrow{u}=(u,w)$ to the IVP (\ref{1.1}) in the case
$(\beta+\mu)=\beta\sigma_{\beta}$ satisfying
\[
\overrightarrow{u}\in C([-T,T]:H^{s}(\mathbb{T})\times H^{s}(\mathbb{T}%
))\text{,\quad}\overrightarrow{u}\in Y_{s}\times Y_{s}\text{,}%
\]
where $c_{0}=$ $(\beta+\mu)\| \overrightarrow{u_0}\|
_{L_{x}^{2}}^{2}$ in the definition of $Y_{s}$.

For each $T^{\prime}\in(0,T)$, there exists $\epsilon>0$ such that
the map $\overrightarrow{{v}_{0}}\mapsto\overrightarrow{v}$ is
Lipschitz continuous from
\[
\left\{  \overrightarrow{{v}_{0}}\in H^{s}(\mathbb{T})\times H^{s}
(\mathbb{T}):\| \overrightarrow{{v}_{0}}-\overrightarrow{{u}
_{0}}\| _{H^{s}}<\epsilon\right\}
\]
to
\[
\left\{  \overrightarrow{v}:\| \overrightarrow{v}-\overrightarrow
{u}\| _{L_{T^{\prime}}^{\infty}H^{s}}+\| \Psi_{T^{\prime}%
}(\overrightarrow{v}-\overrightarrow{u})\| _{Y_{s}}<\infty\right\}
\text{.}%
\]

\end{theorem}

We notice that the solution flow of (\ref{1.1}) is invariant by
the following quantities in the case $\sigma_{\alpha}
=\sigma_{\beta}=\sigma_{\mu}=1$.
%Actually these conserved
%quantities are the same as the ones obtained in \cite{4a} and we
%repeat it here for convenience of the reader.
\begin{equation}\label{1.9}
I_{1}(u,w)=\|u\| _{L^{2}\left(  \Omega\right)  } ^{2}+\| w\|
_{L^{2}\left(  \Omega\right)  }^{2}=I_{1} (u_{0},w_{0})
\end{equation}
and
\begin{equation}\label{1.10}
\begin{split}
I_{2}(u,w)=&i(-3\gamma\alpha\!+\!\beta q\!+\!2\mu q)
\int_{\Omega}(u\overline{u}_{x}+w\overline{w}_{x})dx
+\frac{3}{2}\gamma\int_{\Omega}(|u_{x}|^2+|w_{x}|^2)\,dx\\
&+\frac{1}{2}\left(\beta+2\mu\right)\int_{\Omega}(|u|^4+|w|^4)\,dx
+\left(\beta+2\mu\right)  \int_{\Omega}  |u|^{2} | w| ^{2}dx\\
=&I_{2}(u_{0},w_{0})\,
\end{split}
\end{equation}
for either $\Omega=\mathbb{R}$ or $\mathbb{T}$, see \cite{4a}.
These quantities allow us to extend our local result globally.
This is what is contained in the next result.

Before stating our main results, let us define the notation that
will be used throughout this work.

\begin{definition}
Let $\mathcal{P}=\mathcal{C}^{\infty}\left(  \mathbb{T}\right)
=\left\{ g:\mathbb{R}\rightarrow\mathbb{C};\text{
}g\in\mathcal{C}^{\infty}\text{periodic with period} 2\pi\right\}
$. $\mathcal{P}^{\prime}$(dual of $\mathcal{P}$) is the collection
of all linear functionals  from $\mathcal{P}$ to $\mathbb{C}$.
$\mathcal{P}^{\prime}$ is periodic distributions. If
$g\in\mathcal{P}^{\prime}$ denote the value of  $g$ in $\varphi$
by $g\left( \varphi\right) =\langle g,\varphi\rangle$. Consider
the functions  $\theta_{n}\left( x\right) =e^{inx}$,
$n\in\mathbb{Z}$ and $x\in\mathbb{R}$. The Fourier transform
$g\in\mathcal{P}^{\prime}$ is the function
$\widehat{g}:\mathbb{Z}\rightarrow\mathbb{C}$ defined by
$\widehat{g}\left( n\right)  =\langle g,\theta_{-n}\rangle$. If
$g$ is periodic of period $2\pi$, for instance, $g\in L^{1}\left(
\mathbb{T}\right)  $ then
\[
\widehat{g}(n)=\int_{\mathbb{T}}e^{-ixn}g(x)dx=\frac{1}{2\pi}\int
\nolimits_{0}^{2\pi}e^{-ixn}g(x)dx\text{,}%
\]
where $\mathbb{T=R}/2\pi\mathbb{Z}$ represents the one dimensional
torus.
\end{definition}

\begin{definition}
Let $s\in\mathbb{R}$, the Sobolev space $H^{s}(\mathbb{T)}$
is the set of all $g\in\mathcal{P}^{\prime}$ such that
\[
\| g\| _{H^{s}(\mathbb{T})}=\left(2\pi\sum\limits_{n\in \mathbb{Z}}(1+| n|
^{2})^{s}| \widehat{g}(n)| ^{2}\right)  ^{\frac{1}{2}}<\infty.
\]
In this work we assume that $g$ is periodic of period $2\pi$.
\end{definition}

\begin{definition}
We denote by $\widetilde{f}$ the Fourier transform of $f$ in
relation to space-time variables
\[
\widetilde{f}(n,\tau)=
{\displaystyle\int\nolimits_{\mathbb{R}}}
\int_{\mathbb{T}}e^{-i(xn+t\tau)}f(x,t)dxdt\text{.}%
\]
The inverse Fourier transform is given by
\[
f(x,t)=\sum\limits_{n\in\mathbb{Z}}e^{inx}
{\displaystyle\int\nolimits_{\mathbb{R}}}
e^{it\tau}\widetilde{f}(n,\tau)d\tau\text{.}%
\]
The Fourier transform of $fgh$, where $f=f(x,t)$, $g=g(x,t)$ and
$h=h(x,t)$ are periodic functions with respect to  the $x$ variable obtained as
\[
\widetilde{fgh}(n,\tau)\!=\!(\widetilde{f}\ast\widetilde{g}\ast\widetilde{h})(n,\tau)
\!=\!\sum_{n_{1}\in\mathbb{Z}}\sum_{n_{2}\in\mathbb{Z}}\iint\nolimits_{\mathbb{R}
^{2}}\widetilde{f}(n_{1},\tau_{1})\widetilde{g}(n_{2},\tau_{2})\widetilde
{h}(n-n_{1}-n_{2},\tau-\tau_{1}-\tau_{2})d\tau_{1}d\tau_{2}\text{.}
\]

\end{definition}

\begin{definition}
Let $\mathcal{V}$ be the space of functions $f$ such that
\begin{itemize}
\item[(i)]  $f:\mathbb{T}\times\mathbb{R}\rightarrow\mathbb{C}$,
\item[(ii)]  $f(x,\cdot)\in\mathcal{S}(\mathbb{R})$ for each $x\in\mathbb{T}$,
\item[(iii)]  $f(\cdot,t)\in C^{\infty}(\mathbb{T)}$ for each $t\in\mathbb{R}$.
\end{itemize}
We define the space $X_{s,b}$ associated to operator
$\partial_{t}-i\frac{q}{2}
\partial_{x}^{2}+\partial_{x}^{3}+c_{0}\partial_{x}$ as the completion of $\mathcal{V}$
with respect to the following norm%
\[
\| f\| _{X_{s,b}}:=\| f\| _{(1,s,b)}=\| \langle
n\rangle^{s}\langle
\tau-n^{3}+\frac{q}{2}n^{2}+c_{0}n\rangle^{b}\widetilde{f}(n,\tau)\|
_{l_{n}^{2}L_{\tau}^{2}}\text{,}
\]
where $\langle n\rangle=(1+|n|^{2})^{\frac{1}{2}}$ and $s$, $b$, $c_{0}\in\mathbb{R}$.

The space $Z_{s,b}$ is the completion of $\mathcal{V}$ with respect to the norm
\[
\| f\| _{Z_{s,b}}:=\| f\| _{(2,s,b)}=\| \langle
n\rangle^{s}\langle
\tau-n^{3}+\frac{q}{2}n^{2}+c_{0}n\rangle^{b}\widetilde{f}(n,\tau)\|
_{l_{n}^{2}L_{\tau}^{1}}\text{,}
\]
and we consider the space $Y_{s}=X_{s,\frac12}\cap Z_{s,0}$ with the norm
\[
\| f\| _{Y_{s}}=\| f\| _{(1,s,%
%TCIMACRO{\U{bd}}%
%BeginExpansion
\frac12
%EndExpansion
)}+\| f\| _{(2,s,0)}\text{.}
\]
\end{definition}

\begin{remark}
Note that for $b>\frac{1}{2}$ we have that  $X_{s,b}\subset C(\mathbb{R}_{t};H^{s}%
(\mathbb{T}))$ and
 for $b=0$, we have $Z_{s,0}\subset$
$C(\mathbb{R}_{t};H^{s}(\mathbb{T}))$ and $Y_{s}\subset C(\mathbb{R}_{t}%
;H^{s}(\mathbb{T}))$.

The space $Y_{s}^{T}=\left\{  f\mid_{\lbrack-T,T]}\,:\,f\in Y_{s}\right\} $
with the norm
\[
\| f\| _{Y_{s}^{T}}=\inf\left\{  \|g\| _{Y_{s}}\;:\;
g\mid_{_{\lbrack-T,T]}}=f\text{\hskip5pt and\hskip5pt}g\in Y_{s}\right\},
\]
satisfies $Y_{s}^{T}\subset C([-T,T];H^{s}(\mathbb{T}))$.
\end{remark}
%%%%%%%%%%%%%%%%%%%%%%%%%%%%%%%%%%%%%%%%%%%%%%%%%%%%
%To simplify the notation we write the IVP (\ref{1.1}) as
%\begin{equation}\label{1.15}
%\begin{cases}
%\partial_{t}\overrightarrow{u}-\frac{q}{2}i\partial_{x}^{2}\overrightarrow
%{u}+\frac{\gamma}{2}\partial_{x}^{3}\overrightarrow{u} +c_0 \partial_{x}\overrightarrow{u}+F(\overrightarrow
%{u})=0\\
%\overrightarrow{u}(x,0)=\overrightarrow{u_0},
%\end{cases}
%\end{equation}
%where
%\[
%\ \overrightarrow{u}(x,t)=\left(u(x,t),w(x,t) \right)  \text{,}\hspace {0.6cm} F(\overrightarrow{u})=\left( %F_{1}(u,w),F_{1}(w,u)
%\right)
%\]
%and
%\begin{equation*}
%F_1(u,w)\!=\!( \beta\!+\!\mu)  |u|^2\partial _{x}u\!+\!\beta\sigma_{\beta}|w|^2\partial_{x}u
%\!-\!i\alpha u(|u|^2\!+\!\sigma_{\alpha}|w| ^{2})\!+\!\mu u^{2}\partial_{x}\overline{u}\!+\!\mu\sigma_{\mu}u\partial_{x}(|w| %^{2}).
%\end{equation*}

Now, we are in position to state the main results of this work.

\begin{theorem}\label{teo1.7}
Suppose that $\frac{q}{3}\notin\mathbb{Z}$. If $s\geq\frac{1}{2}$
and $\overrightarrow{u_0}=\left(  u_{0},w_{0}\right)  \in H^{s}(\mathbb{T}%
)\times H^{s}(\mathbb{T})$, then there exist $T(\| \overrightarrow{u}%
_{0}\| _{H^{s}})>0$ and a unique solution
$\overrightarrow{u}=(u,w)$ to the IVP (\ref{1.1}) in the case
$(\beta+\mu)=\beta\sigma_{\beta}$ satisfying
\[
\overrightarrow{u}\in C([-T,T]:H^{s}(\mathbb{T})\times H^{s}(\mathbb{T}%
))\text{,\quad}\overrightarrow{u}\in Y_{s}\times Y_{s}\text{,}%
\]
where $c_{0}=$ $(\beta+\mu)\|
\overrightarrow{u_0}\| _{L_{x}^{2}}^{2}$ in the
definition of $Y_{s}$.

For each $T^{\prime}\in(0,T)$, there exists $\epsilon>0$ such that
the map $\overrightarrow{{v}_{0}}\mapsto\overrightarrow{v}$ is
Lipschitz continuous from
\[
\left\{  \overrightarrow{{v}_{0}}\in H^{s}(\mathbb{T})\times H^{s}
(\mathbb{T}):\| \overrightarrow{{v}_{0}}-\overrightarrow{{u}
_{0}}\| _{H^{s}}<\epsilon\right\}
\]
to
\[
\left\{  \overrightarrow{v}:\|
\overrightarrow{v}-\overrightarrow
{u}\| _{L_{T^{\prime}}^{\infty}H^{s}}+\| \Psi_{T^{\prime}%
}(\overrightarrow{v}-\overrightarrow{u})\|
_{Y_{s}}<\infty\right\}
\text{.}%
\]

\end{theorem}

Combining the local well-posedness result with the conservation
laws the following global result follows.

\begin{theorem}\label{teo1.8}
Let $\overrightarrow{u_0}=\left(  u_{0},w_{0}\right)  \in H^{1}%
(\mathbb{T)}\times H^{1}(\mathbb{T)}$. Then there exists a unique
solution $\overrightarrow{u}=\left( u,w\right) $ to the problem
(\ref{1.1}) with $\sigma_{\alpha}=\sigma_{\beta}=\sigma_{\mu}=1$
and $\mu=0$ satisfying
\[
\left(  u,w\right)  \in C(\mathbb{R};H^{1}(\mathbb{T)\times}H^{1}%
(\mathbb{T))}\text{.}%
\]
\end{theorem}
This work is organized as follows. In second section we will list
a series of estimates in the spaces defined on Definition 1.4 that
will be needed in the proof of Theorem \ref{teo1.7}.
 In the third section we establish local well-posedness for the periodic initial value problem
associated to \eqref{1.1} for data in $H^{s}(\mathbb{T})\times
H^s(\mathbb{T})$, $s\geq\frac{1}{2}$ and the fourth section is
dedicated to global result. We finish the paper with some comments
about future work.

\section{Preliminary estimates}

To prove our periodic results we use the spaces introduced by Bourgain  in
\cite{4}, the contraction principle and also the properties of the
solutions to the linear problem
\begin{equation}\label{1.7}
\begin{cases}
\partial_{t}u-i\frac{q}{2}\partial_{x}^{2}u+\partial_{x}^{3}u+c_{0}\partial_{x}u=0,
\;\;x\in\mathbb{T}, \;t>0,\\
\partial_{t}w-i\frac{q}{2}\partial_{x}^{2}w+\partial_{x}^{3}w+c_{0}\partial_{x}w=0,\\
u(x,0)=u_0(x)\;\;\text{and}\;\;w(x,0)=w_0(x).
\end{cases}
\end{equation}
This linear system differs from the one used in \cite{4a} because
of the terms containing $c_{0}\partial_{x}u$ and
$c_{0}\partial_{x}w$, where the constant $c_{0}=\left(
\beta+\mu\right) [\| u\| _{L_{x}^{2}}^{2}+\| w\|
_{L_{x}^{2}}^{2}]$.

\begin{remark}
The problem (2.1) is a particular case of (1.1) with $\gamma=2$.
So, we assume  $\gamma=2$ without loss of generality because if
$\gamma\neq2$, it is enough to consider the change of variable
$v_{1} (x,t)=u(\theta x,t)$ and $v_{2}(x,t)=w(\theta x,t)$, where
$\theta=\sqrt[3] {\frac{\gamma}{2}}$, obtaining
$u(x,t)=v_{1}(\theta^{-1}x,t)$, $w(x,t)=v_{2} (\theta^{-1}x,t)$,
$\frac{\gamma}{2}\partial_{x}^{3}u=\partial_{x}^{3}v_{1}$ and
$\frac{\gamma}{2}\partial_{x}^{3}w=\partial_{x}^{3}v_{2}$.

We note that $\left\Vert u(t)\right\Vert _{L_{x}^{2}}
^{2}+\left\Vert w(t)\right\Vert _{L_{x}^{2}}^{2}$ is a conserved
quantity, see (1.4). Therefore the constant $c_{0}$ is independent
of $t$. In what follows  this constant is used as the number $c_0$
in the Definition 1.4. It plays important role to get the bounds
we need.
%as ahead in the definition of $\phi(n)$.

\end{remark}

%The necessity of these terms will be clear later when the constant
%$c_{0}=\left( \beta+\mu\right) [\| u\| _{L_{x}^{2}}^{2}+\| w\|
%_{L_{x}^{2}}^{2}]$ will appear.
%Here $H^{s}(\mathbb{T})$ denotes the Sobolev space of order
%$s$ over $\mathbb{T}$.  More precisely, $H^{s}(\mathbb{T})$  denotes  the space of all periodic
%distributions $f$ with period $2\pi$ such that
%\[
%\| f\| _{H^{s}(\mathbb{T})}=\left(2\pi\sum\limits_{n\in \mathbb{Z}}(1+| n|^{2})^{s}
%| \widehat{f}(n)| ^{2}\right)  ^{\frac{1}{2}}<\infty.
%\]

The solution of \eqref{1.7} is given by the unitary group
$\left\{  W_{p}(t)\right\}  _{t\in\mathbb{R}}$ in $H^{s}(\mathbb{T)\times}H^{s}(\mathbb{T)}$,
defined as
\begin{equation}\label{1.8}
\overrightarrow{u}(x,t)=W_{p}(t)\overrightarrow{u_0}=(S_{p}(t)u_{0},S_{p}(t)w_{0}),
\end{equation}
where, the subscript $p$ only means "periodic", and
\begin{equation*}
S_{p}(t)u_{0}=\sum\limits_{n\in\mathbb{Z}}e^{inx}e^{it\phi(n)}\widehat{u}_{0}\left(  n\right),
\end{equation*}
and $\phi(n)=n^{3}-\frac{q}{2}n^{2}-c_{0}n$.

Let $q_{\pm}(n,\tau)=\tau-n^{3}\pm\frac{q}{2}n^{2}+c_{0}n$, then
we obtain the following equalities,
\begin{align}\label{4.4}
\langle q_{-}(n,\tau)\rangle^{b}| \widetilde{\overline{u}}%
(n,\tau)| =\langle q_{+}(-n,-\tau)\rangle^{b}|
\widetilde{u}(-n,-\tau)| \text{,}
\\
\| f\| _{(1,s,b)}=\| \langle n\rangle^{s}\langle
q_{+}(n,\tau)\rangle^{b}\widetilde{f}(n,\tau)\| _{l_{n}^{2}L_{\tau
}^{2}}\text{.}
\end{align}
%There is no available regularizing effects in the periodic case.
The main linear estimates are the following.

\begin{lemma}\label{2.1}
For $s\in\mathbb{R}$ we have
\[
\| \Psi(t)W_{p}(t)\overrightarrow{u_0}\| _{(1,s,\frac12)}\leq c\| \overrightarrow{u_0}\| _{H^{s}
\times H^{s}
}\text{,}
\]
\[
\| \Psi(t)W_{p}(t)\overrightarrow{u_0}\|
_{(2,s)}\leq
c\| \overrightarrow{u_0}\| _{H^{s}\times H^{s}}\text{.}%
\]
Therefore
\begin{equation}\label{4.6}
\| \Psi(t)W_{p}(t)\overrightarrow{u_0}\|
_{Y_{s}\times
Y_{s}}\leq c\| \overrightarrow{u_0}\| _{H^{s}\times H^{s}%
}\text{,}
\end{equation}

where $\Psi(t)W_{p}(t)\overrightarrow{u_0}$ is given by
\[
\Psi(t)W_{p}(t)\overrightarrow{u_0}=(\psi(t)S_{p}(t)u_{0},\psi
(t)S_{p}(t)w_{0})
\]
and $\psi$ connotes a cut-off function satisfying $\psi=1$ in $[-1,1]$,
$\psi\in C_{0}^{\infty}$ and
$\mathit{supp}\,\psi\subseteq(-2,2)$.
\end{lemma}
\begin{proof}
Taking in account the Definition 1.4, to obtain (\ref{4.6}) it is
enough to estimate
\begin{equation}\label{4.8}
\begin{split}
\| \psi(t)S_{p}(t)u_{0}\| _{(1,s,\frac12
)}^{2}  &  =\| \langle n\rangle^{s}\langle
q_{+}(n,\tau)\rangle
^{\frac{1}{2}}\widehat{\psi}(q_{+}(n,\tau))\widehat{u}_{0}(n)\|
_{l_{n}^{2}L_{\tau}^{2}}^{2}\\
&  =\sum\limits_{n\in\mathbb{Z}}\langle n\rangle^{2s}| \widehat
{u}_{0}(n)| ^{2}\int_{\mathbb{R}}\langle q_{+}(n,\tau)\rangle
^{1}| \widehat{\psi}(q_{+}(n,\tau))| ^{2}d\tau\\
&  \leq c\| u_{0}\| _{H^{s}}^{2}\|
\psi\| _{H^{1/2}}^{2}\leq c\| u_{0}\|
_{H^{s}}^{2}
\end{split}
\end{equation}
and
\begin{equation}\label{4.9}
\begin{split}
\| \psi(t)S_{p}(t)u_{0}\| _{(2,s,0)}^{2}  &
=\sum\limits_{n\in\mathbb{Z}}\langle n\rangle^{2s}| \widehat
{u}_{0}(n)| ^{2}\left( \int_{\mathbb{R}}|
\widehat{\psi}(q_{+}(n,\tau))
| d\tau\right)  ^{2}\\
& \leq c\| u_{0}\| _{H^s}^2.
\end{split}
\end{equation}
Details of the computations are found in \cite{27}.
\end{proof}

%\begin{lemma}\label{2.2}
%Let $\frac{-1}{2} < b^{'} \leq 0 \leq b \leq b^{\prime}$ and ${T \in [0,1]}$. Then for
%$ F \in X_{s,b^{\prime}}(\phi)$ we have
%\| \Psi(t)W_{p}(t)u_0\| _{X_{s,b}}\leq c\| {u_0}\| _{H^{s}}
%\]
%\begin{equation}
%\| \Psi(t)W_{p}(t)w_0\| _{X_{s,b}}\leq c\| w_0\| _{H^{s}}
%\end{equation}

\section{Proof of Theorem 1.6}

The result in this section requires a new set of computations, so
we will present it in more detailed setting. To simplify the notation we write (\ref{1.1}) as
\begin{equation}\label{4.1}
\left\{
\begin{array}
[c]{l}%
\partial_{t}\overrightarrow{u}-\frac{q}{2}i\partial_{x}^{2}\overrightarrow
{u}+\partial_{x}^{3}\overrightarrow{u}+{\bf{c}}_{0}\partial_{x}\overrightarrow
{u}= G(\overrightarrow{u}),\\
\overrightarrow{u}(x,0)=\overrightarrow{u_0}\in
H^{s}(\mathbb{T})\times H^{s}(\mathbb{T}),
\end{array}
\right.
\end{equation}
where
\[
G(\overrightarrow{u})=\left(G_{1}(u,w),G_{1}(w,u)\right),
\]
with
\begin{align}\label{4.2}
 G_{1}(u,w)  &  =\left(  \beta+\mu\right)  \left [| u| ^{2}-\|
u\| _{L_{x}^{2}}^{2}-\| w\| _{L_{x}^{2}}^{2}\right ]\partial_{x}u
+\beta\sigma_{\beta} | w|
^{2}\partial_{x}u\\
&  \mu u^{2}\partial_{x}\overline{u}+\mu\sigma_{\mu}u\partial_{x}(|
w| ^{2})-i\alpha u(| u|
^{2}+\sigma_{\alpha }| w|
^{2})\text{.}\nonumber
\end{align}

The integral equation associated to (\ref{4.1}) is
\begin{equation}\label{4.3}
\Phi_{\overrightarrow{u_0}}(\overrightarrow{u})=W_{p}(t)\overrightarrow
{u}_{0}-\int\nolimits_{0}^{t}W_{p}(t-t^{\prime})G(\overrightarrow
{u})(t^{\prime})dt^{\prime}\text{,}
\end{equation}
where $W_{p}(t)$ is defined in (\ref{1.8}).

Considering the cut-off function $\psi\in C_{0}^{\infty}$ defined on Lemma \ref{2.1} we
have the following estimate.
%such that
%$\psi=1$ in $[-1,1]$ and $\mathit{supp}\psi\subseteq(-2,2)$.
%Define $\Psi \overrightarrow{u}=\left( \psi u,\psi w\right)$ and
%for $a\in\mathbb{R}$, $a\neq0$, $\psi_{a}(t)=\psi\left(
%\frac{t}{a}\right)$.

\begin{lemma}
For $s\in\mathbb{R}$ we have

\begin{equation}\label{4.7}
\begin{split}
&\| \Psi(t)\int_{0}^{t}\!W_{p}(t-t^{\prime})G(\overrightarrow{u})(t')dt'\|_{Y_{s}\times Y_{s}}
  \leq c\| G_{1}(u,w)\| _{(1,s,-\frac12)}\!\\
 &+\!c\| G_{1}(w,u)\| _{(1,s,-\frac12)}
+c\left(  \sum\limits_{n\in\mathbb{Z}}\langle n\rangle^{2s}\left(
\int_{-\infty}^{+\infty}\frac{\widetilde{G_{1}(u,w)}(n,\tau)}{\langle
q_{+}(n,\tau)\rangle}d\tau\right)  ^{2}\right)  ^{\frac{1}{2}}.
 %&+c\left(  \sum\limits_{n\in\mathbb{Z}}\langle
%n\rangle^{2s}\left(
%\int_{-\infty}^{+\infty}\frac{\widetilde{G_{1}(w,u)}(n,\tau)}{\langle
%q_{+}(n,\tau)\rangle}d\tau\right)  ^{2}\right)  ^{\frac{1}{2}}\text{.}
\end{split}
\end{equation}
\end{lemma}

\begin{proof}
%According to Definition 1.4, to obtain (\ref{4.6}) it is enough to estimate
%\begin{split}
%\| \psi(t)S_{p}(t)u_{0}\| _{(1,s,\frac12
%)}^{2}  &  =\| \langle n\rangle^{s}\langleq_{+}(n,\tau)\rangle
%^{\frac{1}{2}}\widehat{\psi}(q_{+}(n,\tau))\widehat{u}_{0}(n)\|
%_{l_{n}^{2}L_{\tau}^{2}}^{2}\\
%&  =\sum\limits_{n\in\mathbb{Z}}\langle n\rangle^{2s}\|
%\widehat {u}_{0}(n)\| ^{2}\int_{\mathbb{R}}\langle
%q_{+}(n,\tau)\rangle ^{1}\| \widehat{\psi}(q_{+}(n,\tau))\| ^{2}d\tau\\
%&  \leq c\| u_{0}\| _{H^{s}}^{2}\|\psi\| _{H^{1/2}}^{2}\leq c\| u_{0}\|_{H^{s}}^{2}
%\end{split}
%\end{equation}
%\begin{equation}\label{4.9}
%\begin{split}
%\| \psi(t)S_{p}(t)u_{0}\| _{(2,s)}^{2}  &
%=\sum\limits_{n\in\mathbb{Z}}\langle n\rangle^{2s}\|
%\widehat {u}_{0}(n)\| ^{2}\left( \int_{\mathbb{R}}\| \widehat{\psi}(q_{+}(n,\tau))\| d\tau\right)  ^{2}\\
%& \leq c\| u_{0}\| _{H^s}^2.
%\end{split}
%\end{equation}

To obtain (\ref{4.7}) it is sufficient to estimate
\begin{equation}\label{4.10}
\begin{split}
\|\psi(t)\int\nolimits_{0}^{t}S_{p}(t-t^{\prime})&G_{1}(u,w)(t')dt'\| _{Y_{s}}
\le c\| G_{1}(u,w)\| _{(1,s,-\frac12)}\\
&+ c\left(\sum\limits_{n\in \mathbb{Z}}\langle n\rangle^{2s}\left(
\int_{-\infty}^{+\infty}
\frac{\widetilde{G_{1}(u,w)}(n,\tau)}{\langle
q_{+}(n,\tau)\rangle} d\tau\right)^{2}\right)^{\frac{1}{2}}.
\end{split}
\end{equation}
From the relation $\int_{0}^{t}h\left(  t^{\prime}\right)  dt^{\prime}%
=\int_{-\infty}^{+\infty}\frac{e^{it\lambda}-1}{i\lambda}\widehat{h}\left(
\lambda\right)  d\lambda$ and $G_{1}(x,t)=G_{1}(u,w)(x,t)$, we obtain
\begin{equation}\label{4.11}
\begin{split}
\psi(t)\int\nolimits_{0}^{t}&S_{p}(t-t^{\prime})G_{1}(x,t^{\prime})dt^{\prime}
=\\
&\psi(t)\sum\limits_{n\in\mathbb{Z}}e^{inx}\int_{-\infty}^{+\infty}\left(
\frac{e^{it\left(  q_{+}(n,\lambda)\right)
}-1}{i(q_{+}(n,\lambda))}\right)
e^{it\left(  n^{3}-\frac{q}{2}n^{2}-c_{0}n\right)  }\widetilde{G}%
_{1}(n,\lambda)d\lambda.
\end{split}
\end{equation}
Let $\varphi\in C_{0}^{\infty}$ be another function such that
$\varphi\equiv1$ in $[-B,B]$, $\mathit{supp}$
$\varphi\subseteq(-2B,2B),$  where
$B<\frac{1}{100}$, say. Using $\varphi$, we can write \eqref{4.11} as
\begin{equation*}
\begin{split}
&\psi(t)\sum\limits_{n\in\mathbb{Z}}e^{inx}\int_{-\infty}^{+\infty}
\left(\frac{e^{it\left(  q_{+}(n,\lambda)\right)}-1}{i(q_{+}(n,\lambda))}\right)
e^{it\phi(n)  }\varphi(q_{+}(n,\lambda))\widetilde{G}_{1}(n,\lambda)d\lambda\\
&
+\psi(t)\sum\limits_{n\in\mathbb{Z}}e^{inx}\int_{-\infty}^{+\infty}\left(
\frac{e^{it\left(  q_{+}(n,\lambda)\right)
}-1}{i(q_{+}(n,\lambda))}\right) e^{it\phi(n) }\left(  1-\varphi
(q_{+}(n,\lambda))\right)  \widetilde{G}_{1}(n,\lambda)d\lambda\\
&  =J_{1}(x,t)+J_{2}(x,t).
\end{split}
\end{equation*}
That is,
\begin{equation*}
\begin{split}
J_1(x,t) =&\psi(t)\sum\limits_{n\in\mathbb{Z}}e^{inx}e^{it\phi(n)}\times\\
&\int_{-\infty}^{+\infty}\sum\limits_{k\geq1}\left(\frac{(q_{+}(n,\lambda))^{k-1}
e^{it\left(q_{+}(n,\lambda)\right)  }}{k!}i^{k-1}t^{k}\right)
\varphi(q_{+}(n,\lambda))\widetilde{G}_{1}(n,\lambda) d\lambda
\end{split}
\end{equation*}
\begin{equation}\label{4.12}
\begin{split}
&J_{2}(x,t) =\psi(t)\sum\limits_{n\in\mathbb{Z}}e^{inx}e^{it\phi(n)}
\int_{-\infty}^{+\infty}\left( \frac{e^{it\left(
q_{+}(n,\lambda)\right)  }}{i(q_{+}(n,\lambda))}\right) \left(
1-\varphi(q_{+}(n,\lambda))\right)  \widetilde{G}_{1}(n,\lambda
)d\lambda\\
&  -\psi(t)\sum\limits_{n\in\mathbb{Z}}e^{inx}e^{it\phi(n)}\int_{-\infty}^{+\infty}
\left(\frac{1-\varphi(q_{+}(n,\lambda)}{i(q_{+}(n,\lambda))}\right)
\widetilde{G}_1(n,\lambda)d\lambda\\
&=J_{2}^{1}(x,t)+J_{2}^{2}(x,t).
\end{split}
\end{equation}
Therefore, $\| \psi(t)\int\nolimits_{0}^{t}S_{p}(t-t^{\prime}%
)G_{1}(u,w)(t^{\prime})dt^{\prime}\|
_{Y_{s}}\leq\| J_{1}\| _{Y_{s}}+\|
J_{2}^{1}\| _{Y_{s}}+\| J_{2}^{2}\|
_{Y_{s}}$.\newline Now, we compute the norms of $ J_{1}, J_{2}^{1}$ and
$J_{2}^{2}$
\begin{eqnarray}\label{4.13}
\| J_{1}\| _{(1,s,\frac12)}  &=& \| \langle n\rangle^{s}\langle
q_{+}(n,\tau)\rangle^{\frac
{1}{2}}\widetilde{J_{1}}(n,\tau)\| _{l_{n}^{2}L_{\tau}^{2}}\nonumber \\
&=& \| \langle n\rangle^{s}\langle q_{+}(n,\tau)\rangle^{\frac{1}{2}%
}\sum\limits_{k\geq1}\frac{h_{k}(n)}{k!}\widehat{\psi_{k}}(q_{+}%
(n,\tau))\| _{l_{n}^{2}L_{\tau}^{2}}\text{,}
\end{eqnarray}
where
\[
h_{k}(n)=\int_{-\infty}^{+\infty}i^{k-1}(q_{+}(n,\lambda))^{k-1}\varphi
(q_{+}(n,\lambda))\widetilde{G}_{1}(n,\lambda)d\lambda
\]
and
\[
\widehat{\psi_{k}}(q_{+}(n,\tau))=\int_{-\infty}^{+\infty}e^{-it\left(
q_{+}(n,\tau)\right)  }\psi(t)t^{k}dt\text{.}
\]
Using the properties of $\varphi$ we estimate
\begin{align}\label{4.15}
|h_{k}(n)|\leq c\| \langle q_{+}(n,\lambda)\rangle^{-\frac{1}{2}}
\widetilde{G}_{1}(n,\lambda)\|_{L_{\lambda}^{2}} =L(n).
\end{align}
After an integration by parts and using properties of $\psi$, we have

%%%%%%%%%%%%%%%%%%%%%%%%%%%%%%%%%%%%%%%%%%%%%%%%%%%%%%%%%%%%%%%%%%%
\begin{align}\label{4.15b}
|h_{k}(n)| \leq c\frac{k^{2}+k+1}{| q_{+}(n,\tau)|^{2}}.
\end{align}
On the other hand, we have
\begin{equation}\label{4.16}
| \widehat{\psi_{k}}(q_{+}(n,\tau))|
\leq\|t^{k}\psi\| _{L^{1}}\leq c.
\end{equation}
It follows from (\ref{4.15})-(\ref{4.16}) that
\begin{equation}\label{4.17}
| \widehat{\psi_{k}}(q_{+}(n,\tau))|(1+|  q_{+}(n,\tau)|^{2}) \leq c(k^{2}%
+k+1)\text{.}
\end{equation}
From (\ref{4.13})-(\ref{4.17}) we obtain
\begin{equation}\label{4.18}
\begin{split}
\| J_{1}\| _{(1,s,\frac12)}
&\leq\Big(\sum_{n\in\mathbb{Z}}\langle n\rangle^{2s}\int_{-\infty}^{+\infty}
\langle q_{+}(n,\tau)\rangle\Big(\sum\limits_{k\geq1}\frac{|h_{k}(n)| }{k!}
|\widehat{\psi_{k}}(q_{+}(n,\tau))|\Big)^{2}d\tau\Big)^{\frac{1}{2}}\\
& \leq c\Big(  \sum_{n\in\mathbb{Z}}\langle n\rangle^{2s}
|L(n)| ^{2}\int_{-\infty}^{+\infty}\frac{1}{\langle q_{+}(n,\tau)\rangle^{3}}
\Big(\sum\limits_{k\geq1}\frac{k^{2}+k+1}{k!}\Big)^{2}d\tau\Big)^{\frac{1}{2}}\\
& \leq c\Big(\sum_{n\in\mathbb{Z}}\langle
n\rangle^{2s}|L(n)|^{2}\Big)^{\frac{1}{2}} =c\|G_{1}(u,w)\|
_{(1,s,-\frac12)}.
\end{split}
\end{equation}
The following estimate follows from Young's inequality and
$|\frac{1-\varphi(x)}{x}| \leq c_{1}(B)\langle
x\rangle^{-\frac12}$.
\begin{equation}\label{4.19}
\begin{split}
\|J_{2}^{1}&\|_{(1,s,\frac12)} =\| \langle n\rangle^{s}\langle
q_{+}(n,\tau)\rangle^{\frac{1}{2}}
\Big(\widehat{\psi}(\cdot)\ast\frac{(1-\varphi(q_{+}(n,\cdot )))
\widetilde{G}_{1}(n,\cdot)}{q_{+}(n,\cdot)}\Big)  (\tau
)\| _{l_{n}^{2}L_{\tau}^{2}}\\
&\leq\Big(  \sum_{n\in\mathbb{Z}}\langle n\rangle^{2s} \|\langle
q_{+}(n,\tau)\rangle\widehat{\psi}(\tau)\| _{L_{\tau}^{1}}^{2}
\|\frac{\left(1-\varphi(q_{+}(n,\tau))\right) \widetilde
{G}_{1}(n,\tau)}{q_{+}(n,\tau)}\|
_{L_{\tau}^{2}}^{2}\Big)^{\frac{1}{2}}\\
&\leq c\Big(  \sum_{n\in\mathbb{Z}}\langle n\rangle^{2s}\|
\frac{\left( 1-\varphi(q_{+}(n,\tau))\right)
\widetilde{G}_{1}(n,\tau
)}{q_{+}(n,\tau)}\| _{L_{\tau}^{2}}^{2}\Big)^{\frac12}\\
&\leq c\Big(\sum_{n\in\mathbb{Z}}\langle n\rangle^{2s} \|\langle
q_{+}(n,\lambda)\rangle^{-\frac{1}{2}}\widetilde{G}_{1}(n,\tau)
\|_{L_{\tau}^2}^2\Big)^{\frac12}.
\end{split}
\end{equation}

To estimate
$\| J_{2}^{2}\| _{(1,s, \frac12 )}$
note that
\begin{equation*}
J_{2}^{2} =-\psi(t)\sum\limits_{n\in\mathbb{Z}}e^{inx}e^{it\left(
n^{3}-\frac{q}{2}n^{2}-c_{0}n\right)
}\widehat{r}(n)=-\psi(t)S(t)r(x)\text{,}
\end{equation*}
where
\begin{equation*}
\widehat{r}(n)=\int_{-\infty}^{+\infty}\left(
\frac{1-\varphi(q_{+}(n,\lambda)}{i(q_{+}(n,\lambda))}\right)  \widetilde{G}_{1}.%(n,\lambda)d\lambda \text{.}
\end{equation*}
Then (\ref{4.8}) and the inequality $| \dfrac{1-\varphi(x)}{x}|
\leq\dfrac{c(B)}{\langle x\rangle^{1}}$ imply
\begin{equation}\label{4.20}
\begin{split}
\| J_{2}^{2}\| _{(1,s,\frac12)}
 &  =\| \psi(t)S_{p}(t)r)\| _{(1,s,\frac12)}\leq\| r\| _{H^{s}}\\
&  \leq c\Big(  \sum\limits_{n\in\mathbb{Z}}\langle n\rangle^{2s}\Big|
\int_{-\infty}^{+\infty}\left(  \frac{1-\varphi(q_{+}(n,\lambda))}{i(q_{+}(n,\lambda))}\right)
\widetilde{G}_{1}(n,\lambda)d\lambda\Big|^{2}\Big)^{\frac{1}{2}}\\
& \leq c\Big(  \sum\limits_{n\in\mathbb{Z}}\langle n\rangle^{2s}\Big|
\int_{-\infty}^{+\infty}\frac{\widetilde{G_{1}}(n,\lambda)}{\langle
q_{+}(n,\lambda)\rangle}d\lambda\Big| ^{2}\Big)^{\frac{1}{2}}.
\end{split}
\end{equation}
Using the definition of $h_{k}(n)$ it is straightforward to obtain
\[
\| h_{k}(n)\| \leq\| \mathcal{\chi}_{[-1,1]}%
(q_{+}(n,\lambda))\widetilde{G_{1}}(n,\lambda)\| _{L_{\lambda}^{1}%
}\leq c\Big\| \frac{\widetilde{G_{1}}(n,\lambda)}{\langle q_{+}%
(n,\lambda)\rangle}\Big\| _{L_{\lambda}^{1}}.
\]
Then,
\begin{equation}\label{4.21}
\begin{split}
\| J_{1}\| _{(2,s,0)}  &  =\| \langle
n\rangle ^{s}\int_{-\infty}^{+\infty}e^{-it\left(
q_{+}(n,\tau)\right)  }\psi
(t)\sum\limits_{k\geq1}\frac{t^{k}}{k!}h_{k}(n)dt\| _{l_{n}
^{2}L_{\tau}^{1}}\\
& \leq c\Big(\sum\limits_{n\in\mathbb{Z}}\langle
n\rangle^{2s}\Big(
\int_{-\infty}^{+\infty}\sum\limits_{k\geq1}\frac{|h_{k}(n)|}{k!}|
\widehat{\psi_{k}}(q_{+}(n,\tau))| d\tau\Big)^{2}\Big)^{\frac{1}{2}}\\
&\leq c\Big(\sum\limits_{n\in\mathbb{Z}}\langle n\rangle^{2s}
\|\frac{\widetilde{G_{1}}(n,\lambda)}{\langle q_{+}(n,\lambda)\rangle }
\| _{L_{\lambda}^{1}}^{2}\Big)^{\frac12}.
\end{split}
\end{equation}
Proceeding analogously to (\ref{4.19}) and (\ref{4.20}), we can show that
\[
\| J_{2}^{1}\| _{(2,s,0)}+\| J_{2}^{2}\| _{(2,s,0)}\leq c\left(
\sum\limits_{n\in\mathbb{Z}}\langle n\rangle ^{2s}\|
\frac{\widetilde{G_{1}}(n,\tau)}{\langle q_{+}(n,\tau)\rangle
}\| _{L_{\tau}^{1}}^{2}\right)  ^{\frac{1}{2}}\text{,}%
\]
which finishes the proof of the lemma.
\end{proof}

Let $n,n_{1},n_{2}\in\mathbb{Z}$,
$\tau,\tau_{1},\tau_{2}\in\mathbb{R}$, and set $n_3=n-n_1-n_2$ and $\tau_3=\tau-\tau_1-\tau_2$,
then
\begin{equation}\label{4.22}
\begin{split}
q_{+}(n,\tau)&-q_{+}(n_1,\tau_1)-q_{-}(n_2,\tau_2)-q_{+}
(n_3,\tau_3)\\
=&-3(n_{1}+n_{2})(n-n_{1})(n-n_{2}-\frac{q}{3}).
\end{split}
\end{equation}
Therefore,
\begin{equation*}
\max\{|q_{+}(n,\tau)|,|q_{+}(n_1,\tau_1)|,|q_{-}(n_2,\tau_2)|,|q_{+}(n_3,\tau_3)|\}
\geq\frac{3}{4}| n_{1}+n_{2}| |n-n_{1}| | n-n_{2}-\frac{q}3|.
\end{equation*}

Now, we define $M_{1}\left(n,n_{1},n_{2}\right)=M_{1}$, $L\left(n,n_{1},n_{2}\right)=L$ and
$M_{2}\left(n,n_{1},n_{2}\right)=M_{2} $ as
\[
M_{1}\left(n,n_{1},n_{2}\right)=\max\left\{  |n-n_{1}| ,| n_{1}+n_{2}|, |n-n_{2}| \right\},
\]
\[
M_{2}\left(n,n_{1},n_{2}\right)=\min\left\{|n-n_1|, | n_1+n_2|, |n-n_2|\right\},
\]
and
\begin{equation*}
L(n,n_1,n_2)=
\begin{cases}
|n-n_{1}|,\;\;\;\,  \text{if}\;\;(|n-n_1|-|n_1+n_2|)(|n-n_1|-|n-n_2|)\leq 0,\\
| n_1+n_2|,\;\; \text{if}\;\;(| n_1+n_2|-|n-n_1|)(
|n_1+n_2|-|n-n_2|)  \leq 0,\\
|n-n_2|,\;\;\;\; \text{if}\;\;(| n-n_2|-|n-n_1|)(|n-n_2|-| n_1+n_2|)\leq 0.
\end{cases}
\end{equation*}
Note that $M_2\leq L\leq M_1$. The following inequalities will be
useful to prove the next lemma. The nonlinear term $G_{1}(u,w)$
defined in (\ref{4.2}) restricted to $\left( \beta+\mu\right)
=\beta\sigma_{\beta}$, writes
\begin{align}\label{4.23}
G_{1}(u,w)  &  =\left(  \beta+\mu\right)  [| u|
^{2}-\| u\| _{L_{x}^{2}}^{2}+|
w|
^{2}-\| w\| _{L_{x}^{2}}^{2}]\partial_{x}u\\
&  +\left(  \mu
u^{2}\partial_{x}\overline{u}+\mu\sigma_{\mu}uw\partial
_{x}\overline{w}\right)
+\mu\sigma_{\mu}u\overline{w}\partial_{x}w-i\alpha u(|
u| ^{2}+\sigma_{\alpha}| w|
^{2})\nonumber\\
&=G_{11}(u,w)+G_{12}(u,w)+G_{13}(u,w)+G_{14}(u,w)\text{.}\nonumber
\end{align}
Denote by  $G_{1j}=G_{1j}(u,w)$, for $j=1,\dots, 4$. Next, we
compute $\widetilde{G_{1j}}(n,\tau)$. To do so, note that using
Parceval's identity,
\begin{equation*}
\begin{split}
\widetilde{| u| ^{2}u_{x}}(n,\tau)  =&
{\displaystyle\int\nolimits_{\mathbb{R}}} e^{-it\tau}\left(
\widehat{u}\ast\widehat{\overline{u}}\ast\widehat{u_{x} }\right)
dt ={\displaystyle\int\nolimits_{\mathbb{R}}}
e^{-it\tau}\sum\limits_{n_{1}\in\mathbb{Z}}in_{1}\widehat{u}(n_{1})\left(
\widehat{u}\ast\widehat{\overline{u}}\right)  (n-n_{1})dt \nonumber\\
=&{\displaystyle\int\nolimits_{\mathbb{R}}}
e^{-it\tau}\sum\limits_{n_{1}\neq n}\sum\limits_{n_{2}\in\mathbb{Z}}
in_{1}\widehat{u}(n_{1})\widehat{\overline{u}}(n_{2})\widehat{u}(n_3)dt +
{\displaystyle\int\nolimits_{\mathbb{R}}}
e^{-it\tau}in\widehat{u}(n)\| u\| _{L_{x}^{2}}^{2}\\
=&
{\displaystyle\int\nolimits_{\mathbb{R}}}
e^{-it\tau}\sum\limits_{n_{1}\neq n}\sum\limits_{n_{2}\neq-n_{1}}
in_{1}\widehat{u}(n_{1})\widehat{\overline{u}}(n_{2})\widehat{u}(n_3)dt\\
&+{\displaystyle\int\nolimits_{\mathbb{R}}}
e^{-it\tau}\sum\limits_{n_{1}\in\mathbb{Z}}in_{1}\widehat{u}(n_{1})\widehat{\overline{u}}(-n_{1})
\widehat{u}(n)dt-{\displaystyle\int\nolimits_{\mathbb{R}}}
e^{-it\tau}in\widehat{u}(n)\widehat{\overline{u}}(-n)\widehat{u}(n)dt\\
& +{\displaystyle\int\nolimits_{\mathbb{R}}}
e^{-it\tau}in\widehat{u}(n)\| u\| _{L_{x}^{2}}^{2}dt.
\end{split}
\end{equation*}
Therefore,
\begin{equation}\label{4.24}
\begin{split}
\widetilde{| u| ^{2}u_{x}}(n,\tau)=&\sum\limits_{(n_{1},n_{2})\in
H_{n}^{1}}
in_{1}\iint\nolimits_{\mathbb{R}^{2}}\widetilde{u}(n_{1},\tau_{1})
\widetilde{\overline{u}}(n_{2},\tau_{2})\widetilde{u}(n_3,\tau_3)d\tau_{1}d\tau_{2}\\
& +\sum\limits_{n_{1}\in\mathbb{Z}}in_{1}\iint\nolimits_{\mathbb{R}^{2}}
\widetilde{u}(n_{1},\tau_{1})\widetilde{\overline{u}}(-n_{1},\tau_{2})
\widetilde{u}(n,\tau_3)d\tau_{1}d\tau_{2}\\
& -\iint\nolimits_{\mathbb{R}^{2}}in\widetilde{u}(n,\tau_{1})\widetilde
{\overline{u}}(-n,\tau_{2})\widetilde{u}(n,\tau_3)d\tau_{1}d\tau_{2}\\
& +{\displaystyle\int\nolimits_{\mathbb{R}}}
e^{-it\tau}in\widehat{u}(n)\| u\| _{L_{x}^{2}}^{2}dt,
\end{split}
\end{equation}
where
$H_{n}^{1}=\left\{  \left(  n_{1},n_{2}\right)  \in\mathbb{Z}^{2}:n-n_{1}%
\neq0,n_{1}+n_{2}\neq0\right\}  \text{.}%
$
The term $\widetilde{|w| ^{2}u_{x}}(n,\tau)$ is computed analogously.

We write
$\widetilde{|u|^2u_{x}}=A+\widetilde{|u|_{L_{x}^{2}}^{2}u_{x}}$
and
$\widetilde{|w|^2u_{x}}=B+\widetilde{|w|_{L_{x}^{2}}^{2}u_{x}}$.
Therefore  $\widetilde{G}_{11}(n,\tau)$ will be defined as
$A(n,\tau)+B(n,\tau)$. Explicitly we have that
$\widetilde{G}_{11}(n,\tau)$:
\begin{equation}
\begin{split}
\widetilde{G}_{11}&(n,\tau)  = c_{1}(\beta,\mu)\sum\limits_{(n_{1},n_{2})\in
H_{n}^{1}}in_{1}\iint\nolimits_{\mathbb{R}^{2}}\widetilde{u}(n_{1},\tau
_{1})\widetilde{\overline{u}}(n_{2},\tau_{2})\widetilde{u}(n_3,\tau_3)d\tau_{1}d\tau_{2}\\
& +c_{1}(\beta,\mu)\sum\limits_{(n_{1},n_{2})\in H_{n}^{1}}in_{1}%
\iint\nolimits_{\mathbb{R}^{2}}\widetilde{u}(n_{1},\tau_{1})\widetilde
{\overline{w}}(n_{2},\tau_{2})\widetilde{w}(n_3,\tau_3)d\tau_{1}d\tau_{2}\\
&  +c_{1}(\beta,\mu)\sum\limits_{n_{1}\in\mathbb{Z}}in_{1}\iint
\nolimits_{\mathbb{R}^{2}}\widetilde{u}(n_{1},\tau_{1})\widetilde{\overline
{u}}(-n_{1},\tau_{2})\widetilde{u}(n,\tau_3)d\tau_{1}d\tau
_{2}\\
&  +c_{1}(\beta,\mu)\sum\limits_{n_{1}\in\mathbb{Z}}in_{1}\iint
\nolimits_{\mathbb{R}^{2}}\widetilde{u}(n_{1},\tau_{1})\widetilde{\overline
{w}}(-n_{1},\tau_{2})\widetilde{w}(n,\tau_3)d\tau_{1}d\tau
_{2}\\
& -c_{1}(\beta,\mu)\iint\nolimits_{\mathbb{R}^{2}}in\widetilde{u}(n,\tau
_{1})\widetilde{\overline{u}}(-n,\tau_{2})\widetilde{u}(n,\tau_3)d\tau_{1}d\tau_{2}\\
& -c_{1}(\beta,\mu)\iint\nolimits_{\mathbb{R}^{2}}in\widetilde{u}(n,\tau_{1})
\widetilde{\overline{w}}(-n,\tau_{2})\widetilde{w}(n,\tau_3)d\tau_{1}d\tau_{2}\\
=& R_{1}(n,\tau)+R_{2}(n,\tau)+R_{3}(n,\tau)+R_{4}(n,\tau)+R_{5}%
(n,\tau)+R_{6}(n,\tau).
\end{split}
\end{equation}

\begin{remark}
The computation of $\widetilde{G}_{13}(n,\tau)$ and $\widetilde{G}_{12}(n,\tau)$ are
similar to $\widetilde{G}_{11}(n,\tau)$.
\end{remark}

We have for $\widetilde{G_{14}}(n,\tau)$:
\begin{equation}
\begin{split}
\widetilde{G_{14}}(n,\tau)  =& c_{4}(\alpha)\sum\limits_{(n_{1},n_{2})\in
H_{n}^{1}}\iint\nolimits_{\mathbb{R}^{2}}\widetilde{u}(n_{1},\tau
_{1})\widetilde{\overline{u}}(n_{2},\tau_{2})\widetilde{u}(n_3,\tau_3)d\tau_{1}d\tau_{2}\\
&  +c_{5}(\alpha,\sigma_{\alpha})\sum\limits_{(n_{1},n_{2})\in H_{n}^{1}}%
\iint\nolimits_{\mathbb{R}^{2}}\widetilde{u}(n_{1},\tau_{1})\widetilde
{\overline{w}}(n_{2},\tau_{2})\widetilde{w}(n_3,\tau_3)d\tau_{1}d\tau_{2}\\
&  +c_{6}(\alpha)\iint\nolimits_{\mathbb{R}^{2}}\| \widetilde{u}%
(\tau_{1})\| _{l_{n}^{2}}^{2}\|
\widetilde{u}(\tau
_{2})\| _{l_{n}^{2}}^{2}\widetilde{u}(n,\tau_3)d\tau_{1}d\tau_{2}\\
& +c_{7}(\alpha,\sigma_{\alpha})\iint\nolimits_{\mathbb{R}^{2}}\|
\widetilde{w}(\tau_{1})\| _{l_{n}^{2}}^{2}\|
\widetilde
{w}(\tau_{2})\| _{l_{n}^{2}}^{2}\widetilde{u}(n,\tau_3)d\tau_{1}d\tau_{2}\\
& +c_{8}(\alpha)\iint\nolimits_{\mathbb{R}^{2}}\widetilde{u}(n,\tau
_{1})\widetilde{\overline{u}}(-n,\tau_{2})\widetilde{u}(n,\tau_3)d\tau_{1}d\tau_{2}\\
& +c_{9}(\alpha)\iint\nolimits_{\mathbb{R}^{2}}\widetilde{w}(n,\tau
_{1})\widetilde{\overline{w}}(-n,\tau_{2})\widetilde{u}(n,\tau_3)d\tau_{1}d\tau_{2}.
\end{split}
\end{equation}

\begin{lemma}
Assume $\frac{q}{3}$ is not an integer and
$0<\theta<\frac{1}{12}$. For $s\geq\frac{1}{2}$ there exists $c>0$
such that
\begin{equation}\label{4.27}
\left(  \sum\limits_{n\in\mathbb{Z}}\langle n\rangle^{2s}\left(
\int
_{-\infty}^{+\infty}\frac{\widetilde{G_{1}(u,w)}(n,\tau)}{\langle q_{+}%
(n,\tau)\rangle}d\tau\right)  ^{2}\right)  ^{\frac{1}{2}}\leq
cf(u,w)\text{,}
\end{equation}
and
\begin{equation}\label{4.28}
\| G_{1}(u,w)\| _{(1,s,-\frac12)}\leq cf(u,w)\text{,}
\end{equation}
where
\begin{equation}\label{4.29}
\begin{split}
f(u&,w)=( \| u\| _{(1,s,\frac12 -\theta)}^{2}+\| w\| _{(1,s,\frac12-\theta)}^2)
 \| u\| _{(1,s,\frac12)}\\
&\qquad+\| u\| _{(1,s,\frac12-\theta)}\| w\| _{(1,s,\frac12-\theta)}\| u\| _{(1,s,\frac12)}\\
&+\left(  \| u\| _{(2,\frac12,0)}^{2}+\| w\| _{(2,\frac12,0)}^{2}\right)  \| u\| _{(1,s,0)}
+\|u\|_{(2,\frac12,0)}\| w\| _{(2,\frac12,0)}\| u\| _{(1,s,0)}.
\end{split}
\end{equation}
\end{lemma}

\begin{proof}
The parameter $\frac{q}{3}$ is not an integer because we need that
the third factor in the right hand side of (\ref{4.22}) never
vanishes. Then we have $|n -n_2- q/3| \sim \langle n -n_2
\rangle$. Note also that $|n -n_1| \sim \langle n -n_1 \rangle$
and $|n_1 +n_2| \sim \langle n_1 +n_2 \rangle$ for $n- n_1 \ne 0 $
and $n_1 +n_2 \ne 0$, respectively.

 From the Cauchy-Schwarz inequality, the left hand
side of (\ref{4.27}) is bounded by
\begin{equation}\label{4.30}
\left(  \sum\limits_{n\in\mathbb{Z}}\langle
n\rangle^{2s}\int_{-\infty
}^{+\infty}\frac{| \widetilde{G_{1}(u,w)}(n,\tau)| ^{2}%
}{\langle q_{+}(n,\tau)\rangle^{2(1-a)}}d\tau\int_{-\infty}^{+\infty}%
\frac{d\tau}{\langle q_{+}(n,\tau)\rangle^{2a}}\right)  ^{\frac{1}{2}}%
\text{,}
\end{equation}
where $a$ will be determined later. Consider first
$\widetilde{G_{1}(u,w)}(n,\tau)=R_{2}(n,\tau)$. In this case, \eqref{4.30} is bounded by
\begin{equation}\label{4.31}
\begin{split}
&\Big( \sum\limits_{n}\sum\limits_{(n_{1},n_{2})\in H_{n}^{1}}\Big( \int_{\mathbb{R}^{3}}
\frac{\langle n\rangle^{2s}| n_1|^2}{\langle q_{+}(n,\tau)\rangle^{2(1-a)}}| \widetilde{u}(n_1,\tau_1)| ^2\\
&\times |\widetilde{\overline{w}}(n_{2},\tau_{2})| ^{2}
|\widetilde{w}(n_3,\tau_3)|^2 d\tau_{1}d\tau_{2}d\tau\Big)  I_{a}\Big)^{\frac{1}{2}}\\
=&\Big( \sum\limits_{n}\sum\limits_{(n_{1},n_{2})\in H_{n}^{1}}
\int_{\mathbb{R}^{3}}\frac{\langle n\rangle^{2s}|n_1|^2}{\langle q_{+}(n,\tau)\rangle^{2(1-a)}}
\frac{g(n_{1},\tau_{1})h(n_{2},\tau_{2})p(n_3,\tau_3)}{\langle
n_{1}\rangle^{2s}\langle n_{2}\rangle^{2s}\langle n_3\rangle^{2s}}\\
&\;\;\; \times\frac{d\tau_{1}d\tau_{2}d\tau}{\langle
q_{+}(n_{1},\tau _{1})\rangle^{1-2\theta}\langle q_{-}(n_{2},\tau_{2})\rangle^{1-2\theta}
\langle q_{+}(n_3,\tau_3)\rangle^{1-2\theta}}I_{a}\Big)  ^{\frac{1}{2}},
\end{split}
\end{equation}
where%
\[
I_{a}=\int_{-\infty}^{+\infty}\frac{d\tau}{\langle q_{+}(n,\tau)\rangle^{2a}%
}\text{,}%
\]%
\[
g(n_{1},\tau_{1})=\langle n_{1}\rangle^{2s}\langle
q_{+}(n_{1},\tau _{1})\rangle^{1-2\theta}|
\widetilde{u}(n_{1},\tau_{1})|
^{2}\text{,}%
\]%
\[
h(n_{2},\tau_{2})=\langle n_{2}\rangle^{2s}\langle
q_{-}(n_{2},\tau _{2})\rangle^{1-2\theta}|
\widetilde{\overline{w}}(n_{2},\tau
_{2})| ^{2}%
\]
and
\begin{equation*}
p(n_3,\tau_3)=\langle n_3\rangle^{2s}\langle q_{+}(n_3,\tau_3)\rangle^{1-2\theta}
|\widetilde{w}(n_3,\tau_3)| ^{2}.
\end{equation*}
To estimate (\ref{4.31}) we divide the region of integration in four
parts:
\[
A_1=\left\{ \left(  \tau_{1},\tau_{2},\tau\right)  :\;| q_{+}(n,\tau)|
\geq\max\left\{| q_{+}(n_{1},\tau_{1})| ,
| q_{-}(n_{2},\tau_2)|, | q_{+}(n_3,\tau_3)|\right\}  \right\}
\text{,}
\]
\[
A_{2}=\left\{  \left(  \tau_{1},\tau_{2},\tau\right)  :| q_{+}(n_{1},\tau_{1})|
\geq\max\left\{|q_{+}(n,\tau )| , |q_{-}(n_{2},\tau_{2})| , |q_{+}(n_3,\tau_3)| \right\}
\right\},
\]%
\[
A_{3}=\left\{  \left(  \tau_{1},\tau_{2},\tau\right)  :\;| q_{-}(n_{2},\tau_{2})|
 \geq\max\left\{|q_{+}(n_{1},\tau _{1})| , |q_{+}(n,\tau)| ,| q_{+}(n_3,\tau_3)| \right\} \right\},
\]%
\[
A_{4}=\left\{  \left(  \tau_{1},\tau_{2},\tau\right)  :\;|q_{+}(n_3,\tau_3)|
\geq\max\left\{ | q_{+}(n_{1},\tau_{1})|
,| q_{-}(n_{2},\tau _{2})| , |q_{+}(n,\tau)| \right\}  \right\}.
\]
We also consider the sum in $n_{1}$ and $n_{2}$ of (\ref{4.31}) in the
following three cases
\begin{equation}\label{4.32}
M_{1}\geq L\geq\frac{| n| }{5}>M_{2}\text{,}
\end{equation}
\begin{equation}\label{4.33}
M_{1}\geq\frac{2| n| }{3}\geq\frac{|n|}{5}>L\geq M_{2}\text{,}
\end{equation}
\begin{equation}\label{4.34}
M_{1}\geq L\geq M_{2}\geq\frac{| n| }{5}
\end{equation}

and $I_{a}$ is bounded as
\begin{equation}\label{4.36}
\begin{split}
\int_{-\infty}^{+\infty}\frac{\langle
n^{2}+q_{+}(n,\tau)\rangle^{2a}d\tau }{\langle
n^{2}+q_{+}(n,\tau)\rangle^{2a}\langle q_{+}(n,\tau)\rangle^{2a}}
& \leq\int_{-\infty}^{+\infty}\frac{c\langle
q_{+}(n,\tau)\rangle^{2a}d\tau
}{\langle n^{2}+q_{+}(n,\tau)\rangle^{2a}\langle q_{+}(n,\tau)\rangle^{2a}}\\
&  \leq c\int_{-\infty}^{+\infty}\frac{d\tau}{\langle n^{2}+q_{+}(n,\tau)\rangle^{2a}}\\
&  \leq\frac{c}{\langle n\rangle^{4a-2}}\text{,\;\;\;\;
for}\;a>\frac{1}{2}.
\end{split}
\end{equation}

The first inequality in (\ref{4.36}) follows from $|
q_{+}(n,\tau)| \geq cn^{2}\langle M_{2}\rangle\geq cn^{2}$ and the
last inequality is a consequence of
\begin{equation}\label{4.37}
\int_{\mathbb{R}}\frac{d\tau}{\langle\tau\rangle^{\alpha}\langle\tau
-\theta\rangle^{\beta}}\leq\frac{c}{\langle\theta\rangle^{d}},
\qquad \text{with}, \qquad d= \min \{\alpha, \beta, \alpha+ \beta
-1\}.
\end{equation}

Then, in the case (\ref{4.32}) and in the region $A_{1}$, we bound (\ref{4.31}) by

\begin{equation}\label{4.35}
\begin{split}
&\underset{n,\tau}{\sup}\Big[ \frac{\langle n\rangle^{s-2a+1}}{\langle q_{+}(n,\tau)\rangle^{(1-a)}}
\Big(\sum\limits_{(n_{1},n_{2})\in A_{n,\tau}}\!\!\frac{|n_1| ^{2}}{\langle n_{1}\rangle ^{2s}
\langle n_{2}\rangle^{2s}\langle n_3\rangle^{2s}} H_{\theta}(n,n_{1},n_{2})\Big) ^{1/2}
\Big]\\
& \hspace{0.5cm} \times c\| u\| _{(1,s,\frac12-\theta)}\| w\| _{(1,s,\frac12-\theta)}^{2}\text{,}
\end{split}
\end{equation}
where
\begin{equation*}
\begin{split}
A_{n,\tau} &  =\left\{  \left(  n_{1},n_{2}\right)  :M_{1}\geq L\geq \frac{| n| }{5}>M_{2},n\neq
n_{1},n_{1}\neq-n_{2},\right.
\\
& \qquad \left.  | q_{+}(n,\tau)| \geq | n_1+n_2| | n-n_{1}| | n-n_{2}-\frac{q}{3}| \right\},
\; \; \; \text{and}
\end{split}
\end{equation*}
\[
H_{\theta}(n,n_{1},n_{2})=\int_{\mathbb{R}^{2}}\frac{d\tau_{1}d\tau_{2}%
}{\langle q_{+}(n_{1},\tau_{1})\rangle^{1-2\theta}\langle
q_{-}(n_{2},\tau _{2})\rangle^{1-2\theta}\langle
q_{+}(n_3,\tau_3)\rangle^{1-2\theta}}
\]

Using the identity (\ref{4.22}) and inequality (\ref{4.37}), we bounded $H_{\theta}%
(n,n_{1},n_{2})$ by
\begin{equation}\label{4.38}
\begin{split}
& \int_{\mathbb{R}}\frac{d\tau_{1}}{\langle q_{+}(n_1,\tau_1)\rangle^{1\!-\!2\theta}\langle q_{+}
(n_{1},\tau_{1})\!-\![q_{+}(n,\tau)\!-\!3(  n_1\!+\!n_2)(n\!-\!n_1)( n\!-\!n_2\!-\!\frac{q}{3})]
\rangle^{1\!-\!4\theta}}\\
&  \leq\frac{c}{\langle q_{+}(n,\tau)-3( n_1+n_2)(n-n_1)(  n-n_2-\frac{q}{3})  \rangle^{1-6\theta}}\\
&  \leq\frac{c}{\langle q_{+}(n,\tau)-3\left(  n_{1}+n_{2}\right)
\left(n-n_{1}\right)  \left(  n-n_{2}-\frac{q}{3}\right)  \rangle^{1-\varepsilon}},
\end{split}
\end{equation}
for $\theta\in\left(  0,\frac{1}{12}\right)  $, and
$\varepsilon\in\left( 6\theta,\frac{1}{2}\right)$.
\vspace{0,5cm}

Then, (\ref{4.35}) is estimate by
\begin{equation}\label{4.39}
\begin{split}
&\underset{n,\tau}{c\sup}\Big[\frac{\langle n\rangle^{s-2a+1}}{\langle q_{+}(n,\tau)\rangle^{(1-a)}}
\Big(\sum\limits_{(n_{1},n_{2})\in A_{n,\tau}}
\frac{|n_1|^2}{\langle n_{1}\rangle^{2s}\langle n_{2}\rangle^{2s}\langle n_3\rangle^{2s}}\\
&\times
\frac{1}{\langle q_{+}(n,\tau)-3(n_{1}+n_{2})(n-n_1)(n-n_2-\frac{q}{3})\rangle^{1-\varepsilon}}
\Big)^{\frac{1}{2}}\Big]\\
&\times\|u\| _{(1,s,\frac12-\theta)}\| w\| _{(1,s,\frac12-\theta)}^{2}\\
&  =\underset{n,\tau}{c\sup}\left(  I_{1}\right) ^{\frac{1}{2}}\|
u\| _{(1,s,\frac12-\theta)}\| w\| _{(1,s,\frac12-\theta)}^{2}\leq
c\|u\| _{(1,s,\frac12 )}\| w\| _{(1,s,\frac12-\theta)}^2,
\end{split}
\end{equation}
where
\begin{equation*}
\begin{split}
I_{1}  &  =\frac{\langle n\rangle^{2s-4a+2}}{\langle q_{+}(n,\tau
)\rangle^{2(1-a)}}\left(  \sum\limits_{(n_{1},n_{2})\in A_{n,\tau}}%
\frac{|n_1| ^{2}}{\langle n_{1}\rangle^{2s}\langle
n_{2}\rangle^{2s}\langle n_3\rangle^{2s}}\right. \\
&  \left. \hspace{0.4cm} \times\frac{1}{\langle q_{+}(n,\tau)-3\left(
n_{1}+n_{2}\right) \left(  n-n_{1}\right)  \left(
n-n_{2}-\frac{q}{3}\right)  \rangle
^{1-\varepsilon}}\right).
\end{split}
\end{equation*}
The proof that  $I_{1}\leq c$ can be found in  \cite{38}, Lemma
4.3. The cases (\ref{4.33}) and (\ref{4.34}) are analogous.

Now, we bound \eqref{4.30} in the region $A_{2}$ with $a>\frac{1}{2}$ and
$\widetilde{G_{1}(u,w)}(n,\tau)=R_{2}(n,\tau)$,
\begin{equation}\label{4.40}
c\left(  \sum\limits_{n\in\mathbb{Z}}\langle n\rangle^{2s}\int
\nolimits_{\mathbb{R}}\frac{| R_{2}(n,\tau)|
^{2}}{\langle q_{+}(n,\tau)\rangle^{2(1-a)}}d\tau\right)
^{\frac{1}{2}}=c\Big\|
\frac{\langle n\rangle^{s}R_{2}(n,\tau)}{\langle q_{+}(n,\tau)\rangle^{(1-a)}}
\Big\| _{l_{n}^{2}L_{\tau}^{2}}.
\end{equation}
By duality, (\ref{4.40}) is equal to
\begin{equation}\label{4.41}
\underset{\| h(n,\tau)\|_{l_{n}^{2}L_{\tau}^{2}}}{c\sup}\Big|
\sum\limits_{n\in\mathbb{Z}}\int\nolimits_{\mathbb{R}}\frac{\langle
n\rangle^{s}R_{2}(n,\tau)\overline{h(n,\tau)}}{\langle q_{+}(n,\tau)\rangle^{(1-a)}}d\tau\Big|
=\underset{\| h(n,\tau)\|_{l_{n}^{2}L_{\tau}^{2}}}{c\sup}|A(h)|.
\end{equation}
Note that
\begin{equation}\label{4.42}
\begin{split}
&|A(h)|=\Big|\sum\limits_{n\in\mathbb{Z}}\int\nolimits_{\mathbb{R}}\frac{\langle n\rangle^{s}R_{2}
(n,\tau)\overline{h(n,\tau)}}{\langle q_{+}(n,\tau)\rangle^{(1-a)}}d\tau\Big|\\
&\leq c\sum\limits_{n\in\mathbb{Z}}\langle n\rangle^{s}\left(
\int\nolimits_{\mathbb{R}}\sum\limits_{(n_{1},n_{2})\in H_{n}^{1}}%
\frac{| n_1| |\overline{h(n,\tau)}|}{\langle q_{+}(n,\tau)\rangle^{(1-a)}}
\iint\nolimits_{\mathbb{R}^{2}}%
\frac{g\left(  n_{1},\tau_{1}\right)  p(n_{2},\tau_{2})}{\langle n_{1}%
\rangle^{s}\langle n_{2}\rangle^{s}\langle n_3\rangle^{s}}\right.\\
&\quad \left.  \times\frac{r(n_3,\tau_3)d\tau_{1} d\tau_{2} d\tau}{\langle
q_{+}(n_{1},\tau_{1})\rangle^{\frac{1}{2}}\langle
q_{-}(n_{2},\tau_{2})\rangle^{\frac{1}{2}-\theta}\langle q_{+}(n_3,\tau_3)
\rangle^{\frac{1}{2}-\theta}}\right) \\
 & \leq c\underset{n_{1},\tau_{1}}{\sup}\left(  \sum\limits_{n\in\mathbb{Z}%
}\langle n\rangle^{s}\int\nolimits_{\mathbb{R}}\sum\limits_{n_{2}}%
\frac{| n_1| |\overline{h(n,\tau)}|}{\langle n_{1}\rangle^{s}\langle q_{+}(n_{1},\tau_{1})\rangle^{\frac{1}{2}}
}\int\nolimits_{\mathbb{R}}\frac{p(n_{2},\tau_{2})}{\langle
n_{2}\rangle
^{s}\langle n_3\rangle^{s}}\right.\\
&\quad\left.  \times\frac{d\tau_{2}d\tau}{\langle
q_{+}(n,\tau)\rangle ^{(1-a)}\langle
q_{-}(n_{2},\tau_{2})\rangle^{\frac{1}{2}-\theta}\langle
q_{+}(n_3,\tau_3)\rangle^{\frac{1}{2}-\theta}}\right)\\
&\quad\times\| g\|
_{_{l_{n_{1}}^{2}L_{\tau_{1}}^{2}}}\| r\|
_{_{l_{n1}^{2}L_{\tau_{1}}^{2}}}\text{,}
\end{split}
\end{equation}
where
\begin{equation*}
\begin{split}
g\left(  n_{1},\tau_{1}\right) &=\langle n_{1}\rangle^{2s}\langle q_{+}%
(n_{1},\tau_{1})\rangle^{\frac{1}{2}}|\widetilde{u}(n_{1},\tau_{1})|, \\
p(n_{2},\tau_{2})&=\langle n_{2}\rangle^{s}\langle q_{-}(n_{2},\tau_{2}%
)\rangle^{\frac{1}{2}-\theta}| \widetilde{\overline{w}}(n_{2},\tau_{2})|
\end{split}
\end{equation*}
and
\begin{equation*}
\begin{split}
r(n_3,\tau_3)=\langle n_3\rangle^{s}\langle q_{+}(n_3
,\tau_3\rangle^{\frac{1}{2}-\theta}|\widetilde{w}(n_3,\tau_3)| .
\end{split}
\end{equation*}
Therefore we can bound (3.39) by
\begin{equation}\label{4.43}
\begin{split}
&  \underset{n_{1},\tau_{1}}{c\sup}\left(  \frac{|n_1|
}{\langle n_{1}\rangle^{s}\langle q_{+}(n_{1},\tau_{1})\rangle^{\frac{1}{2}}%
}\left(  \sum\limits_{(n,n_{2})\in D_{n_{1},\tau_{1}}}\int
\nolimits_{\mathbb{R}^{2}}\frac{\langle n\rangle^{2s}}{\langle n_{2}%
\rangle^{2s}\langle n_3\rangle^{2s}}\right.  \right. \nonumber\\
& \hspace{0.4cm} \left.  \left.  \times\frac{d\tau_{2}d\tau}{\langle
q_{+}(n,\tau )\rangle^{2(1-a)}\langle
q_{-}(n_{2},\tau_{2})\rangle^{1-2\theta}\langle
q_{+}(n_3,\tau-\tau_3)\rangle^{1-2\theta}}\right)
^{\frac{1}{2}}\right) \\
& \hspace{0.4cm} \times\| g\|
_{_{l_{n_{1}}^{2}L_{\tau_{1}}^{2}}}\| r\|
_{_{l_{n1}^{2}L_{\tau_{1}}^{2}}}\| p\|
_{_{l_{n_{2}}^{2}L_{\tau_{2}}^{2}}}\|
\overline{h}\|
_{_{l_{n}^{2}L_{\tau}^{2}}}\text{,}\nonumber
\end{split}
\end{equation}
where
\begin{equation*}
\begin{split}
D_{n_{1},\tau_{1}}  &  =\left\{  \left(  n,n_{2}\right)  :n\neq n_{1},n_{1}\neq-n_{2},\right. \\
&  \left.  | q_{+}(n_{1},\tau_{1})|
\geq  |n_{1}+n_{2}| | n-n_{1}| | n-n_2-\frac{q}{3}| \right\}.
\end{split}
\end{equation*}
It follows from (\ref{4.40})-(\ref{4.43}) that (\ref{4.30}) is
bounded by
\begin{equation*}
\begin{split}
&  \underset{n_{1},\tau_{1}}{c\sup}\left(  \frac{|n_1|
}{\langle n_{1}\rangle^{s}\langle q_{+}(n_{1},\tau_{1})\rangle^{\frac{1}{2}}%
}\left(  \sum\limits_{(n,n_{2})\in D_{n_{1},\tau_{1}}}\int
\nolimits_{\mathbb{R}^{2}}\frac{\langle n\rangle^{2s}}{\langle n_{2}%
\rangle^{2s}\langle n_3\rangle^{2s}}\right.  \right. \\
&\hspace{0.4cm} \left.  \left.  \times\frac{d\tau_{2}d\tau}{\langle
q_{+}(n,\tau )\rangle^{2(1-a)}\langle
q_{-}(n_{2},\tau_{2})\rangle^{1-2\theta}\langle
q_{+}(n_3,\tau_3)\rangle^{1-2\theta}}\right)
^{\frac{1}{2}}\right) \\
&\hspace{0.4cm}  \times c\| u\| _{(1,s,\frac12
)}\| w\| _{(1,s,\frac12-\theta)}^{2}\\
&  =\underset{n_{1},\tau_{1}}{c\sup}\left(  I_{2}\right)  ^{\frac{1}{2}
}\| u\| _{(1,s,\frac12)}\| w\| _{(1,s,\frac12
-\theta)}^{2}\leq c\| u\| _{(1,s,\frac12
)}\| w\| _{(1,s,\frac12-\theta)}^{2},
\end{split}
\end{equation*}
where
\begin{align*}
I_{2}  &  =\frac{\| n_{1}\| ^{2}}{\langle
n_{1}\rangle
^{2s}\langle q_{+}(n_{1},\tau_{1})\rangle^{1}}\left(  \sum\limits_{(n,n_{2})\in
D_{n_{1},\tau_{1}}}\int\nolimits_{\mathbb{R}^{2}}\frac{\langle
n\rangle^{2s}}{\langle n_{2}\rangle^{2s}\langle n_3\rangle^{2s}
}\right. \\
&\quad  \left.  \times\frac{d\tau_{2}d\tau}{\langle
q_{+}(n,\tau)\rangle
^{2(1-a)}\langle q_{-}(n_{2},\tau_{2})\rangle^{1-2\theta}\langle q_{+}%
(n_3,\tau_3)\rangle^{1-2\theta}}\right)  \text{.}%
\end{align*}

The proof that $I_{2}\leq c$ is similar to Lemma 4.3 in \cite{38}.
The case $\widetilde{G_{1}(u,w)}(n,\tau)=R_{1}(n,\tau)$ is similar to (\ref{4.40}). In the case
$\widetilde{G_{1}(u,w)}(n,\tau)=R_{4}(n,\tau)$, (\ref{4.30}) is bounded by
\begin{equation}\label{4.44}
\begin{split}
& c\Big\{\sum\limits_{n\in\mathbb{Z}}\langle n\rangle^{2s}\int\nolimits_{\mathbb{R}}\!\!
\Big( \sum\limits_{n_{1}\in\mathbb{Z}}\iint\nolimits_{\mathbb{R}^{2}}|n_1| |\widetilde{u}(n_1,\tau_1)|
|\widetilde{\overline{w}}(-n_1,\tau_2)| |\widetilde{w}(n,\tau_3)| d\tau_1d\tau_2\Big) ^{2}d\tau\Big\} ^{1/2}\\
&\leq c\Big\{  \sum\limits_{n\in\mathbb{Z}}\langle n\rangle^{2s}
\Big( \sum\limits_{n_{1}\in\mathbb{Z}}\int\nolimits_{\mathbb{R}}\Big(
\iint\nolimits_{\mathbb{R}^{2}}\!\!| n_1||\widetilde{u}(n_1,\tau_1)| | \widetilde{\overline{w}}(-n_{1},\tau_{2})|
|\widetilde{w}(n,\tau_3)|d\tau_{1}d\tau_{2}\Big) ^{2}d\tau\Big)\Big\} ^{1/2}.
\end{split}
\end{equation}
Using Minkowski's inequality, (\ref{4.44}) is bounded by
\begin{align*}
&  \left(  \sum\limits_{n\in\mathbb{Z}}\langle n\rangle^{2s}\int
\nolimits_{\mathbb{R}}|\widetilde{w}(n,\tau_3)| ^{2}d\tau\right) ^{\frac{1}{2}}
\left(  \sum\limits_{n_{1}\in\mathbb{Z}}\iint\nolimits_{\mathbb{R}^{2}}|n_1|
 |\widetilde{u}(n_1,\tau_1)| | \widetilde{\overline{w}}(-n_{1},\tau_{2})| d\tau_{1}d\tau_{2}\right) \\
&  \leq c\left(  \sum\limits_{n\in\mathbb{Z}}\langle n\rangle^{2s}
\int\nolimits_{\mathbb{R}}|\widetilde{w}(n,\tau_3)| ^{2}d\tau\right)  ^{\frac{1}{2}}
\left(  \sum\limits_{n_{1}\in\mathbb{Z}}\left[  \int\nolimits_{\mathbb{R}}\langle n_{1}\rangle^{\frac {1}{2}}|
\widetilde{u}(n_{1},\tau_{1})| d\tau_{1}\right]
\right. \\
& \hspace{0.4cm} \left.  \times\left[  \int\nolimits_{\mathbb{R}}\langle
n_{1}\rangle ^{\frac{1}{2}}|\widetilde{\overline{w}}(-n_{1},\tau_{2})|
d\tau_{2}\right]  \right) \\
&  \leq c\| \langle
n\rangle^{s}\widetilde{w}(n,\tau)\|
_{_{l_{n}^{2}L_{\tau}^{2}}}\| \langle n_{1}\rangle^{\frac{1}{2}%
}\widetilde{u}(n_{1},\tau)\| _{_{l_{n_{1}}^{2}L_{\tau}^{1}}%
}\| \langle n_{1}\rangle^{\frac{1}{2}}\widetilde{w}(n_{1}%
,\tau)\| _{_{l_{n_{1}}^{2}L_{\tau}^{1}}}\\
&  =c\| w\| _{(1,s,0)}\| u\| _{(2,\frac12)}\| w\| _{(2,\frac12
)}\text{.}
\end{align*}
The case $\widetilde{G_{1}(u,w)}(n,\tau)=R_{3}(n,\tau)$ is similar
to $R_{4}(n,\tau)$, for $w=u$. In the case of $\widetilde{G_{1}(u,w)}(n,\tau)=R_{6}(n,\tau)$ we bound (\ref{4.30}) by
\begin{equation*}
\begin{split}
&  c\left(  \sum\limits_{n\in\mathbb{Z}}\langle n\rangle^{2s}\int
\nolimits_{\mathbb{R}}\left(
\iint\nolimits_{\mathbb{R}^{2}}| n| |\widetilde{u}(n,\tau_{1})|
|\widetilde{\overline{w}}(-n,\tau_{2})|
|\widetilde {w}(n,\tau-\tau_{1}-\tau_{2})| d\tau_{1}d\tau_{2}\right)  ^{2}%
d\tau\right)  ^{\frac{1}{2}}\\
&  \leq c\left(  \sum\limits_{n\in\mathbb{Z}}\left[  \iint
\nolimits_{\mathbb{R}^{2}}| n| |\widetilde {u}(n,\tau_{1})|
|\widetilde{\overline{w}}(-n,\tau_{2})| d\tau_{1}d\tau_{2}\right]^{2}\right)^{\frac{1}{2}}
\left(\sum\limits_{n\in\mathbb{Z}}\langle n\rangle^{2s}
\int\nolimits_{\mathbb{R}}|\widetilde{w}(n,\tau_3)| ^{2}d\tau\right)^{\frac{1}{2}}\\
&  \leq c\| \langle n\rangle^{s}\widetilde{w}(n,\tau)\|_{_{l_{n}^{2}L_{\tau}^{2}}}
\|\langle n\rangle^{\frac{1}{2}}\widetilde{u}(n,\tau_{1})\|_{_{l_{n}^{2}L_{\tau_{1}}^{1}}}
\| \langle n\rangle^{\frac{1}{2}}\widetilde{w}(n,\tau_{2})\|_{_{l_{n}^{2}L_{\tau_{2}}^{1}}}\\
&  =c\| w\| _{(1,s,0)}\| u\| _{(2,\frac12,0)}\| w\| _{(2,\frac12,0)}.
\end{split}
\end{equation*}
For the cases $\widetilde{G_{1}(u,w)}(n,\tau)=R_{i}(n,\tau)$,
$i=7,...,12$, we follow a similar argument.

The proof of (\ref{4.28}) follows the same lines as (\ref{4.27}), choosing
$a=\frac{1}{2}$ and not considering $I_{a}$, because
\begin{equation}\label{4.45}
\| G_{1}(u,w)\| _{(1,s,- \frac12 )}=\left(
\sum\limits_{n\in\mathbb{Z}}\langle n\rangle^{2s}\int_{-\infty
}^{+\infty}\frac{| \widetilde{G_{1}(u,w)}(n,\tau)| ^{2}%
}{\langle q_{+}(n,\tau)\rangle}d\tau\right)  ^{\frac{1}{2}}\text{.}%
\end{equation}
\end{proof}
The following lemma is found in \cite{38}, Lemma 4.6.
\begin{lemma} For  $s\in\mathbb{R}$, $0<\varepsilon<\frac{1}{2}$, $T\in\left(
0,1\right)$ and $0<\theta^{\prime}<\theta<\frac{1}{2}$, we have the following
inequalities%
\[
\| \psi_{T}(t)f\| _{(1,s, \frac12)}
\leq c(\varepsilon)T^{-\varepsilon}\| f\| _{(1,s,\frac12
)}\text{,}
\]
\[
\| \psi_{T}(t)f\| _{(1,s,
\frac12
-\theta)}\leq cT^{\theta-\theta^{\prime}}\| f\| _{(1,s,\frac12
)}\text{,}%
\]%
\[
\| \psi_{T}(t)f\| _{(2,s,0)}\leq c\|
f\|
_{(2,s,0)}\text{.}%
\]
\end{lemma}
Now we are able to prove the main theorem of this section.

\textit{Proof of Theorem 1.6}.
Let $s\geq\frac{1}{2}$. With all the estimates at hand, for $T\in\left(0,1\right) $,
define%
\[
Y_{T}^{a}=\left\{  \overrightarrow{v}\in Y_{s}\times
Y_{s}:\text{ \ }\| \overrightarrow{v}\|
_{Y_{s}\times Y_{s}}\leq
a\right\}  \text{,}%
\]%
\begin{align}\label{4.46}
\Phi(\overrightarrow{u})  &  =\Psi(t)W_{p}(t)\overrightarrow{{u}_{0}}
-\Psi(t)\int\nolimits_{0}^{t}W_{p}(t-t^{\prime})G(\Psi_{T}(t^{\prime
})\overrightarrow{u}(t^{\prime}))dt^{\prime}\\
&  =\left(
\begin{array}
[c]{c}%
\psi(t)S_{p}(t)u_{0}-\psi(t)\int\nolimits_{0}^{t}S_{p}(t-t^{\prime})G_{1}(\psi
_{T}u,\psi_{T}w)(t^{\prime})dt^{\prime}\\
\psi(t)S_p(t)w_{0}-\psi(t)\int\nolimits_{0}^{t}S_p(t-t^{\prime})G_{1}(\psi
_{T}w,\psi_{T}u)(t^{\prime})dt^{\prime}%
\end{array}
\right)\text{.} \nonumber
\end{align}
From the lemmata 4.2 and 4.5 we have that
\begin{align*}
\| \Phi(\overrightarrow{u})\| _{Y_{s}\times
Y_{s}} &  \leq c\| \overrightarrow{{u}_{0}}\|
_{H^{s}\times H^{s}}+f(\psi
_{T}u,\psi_{T}w)\\
&  \leq c\| \overrightarrow{{u}_{0}}\| _{H^{s}\times H^{s}%
}+c\left(  \| \psi_{T}u\| _{(1,s,\frac12-\theta)}^{2}+\| \psi_{T}w\| _{(1,s,\frac12
-\theta)}^{2}\right)  \| \psi_{T}u\| _{(1,s,\frac12)}\\
&\hspace{0.4cm}  +c\| \psi_{T}u\| _{(1,s,\frac12-\theta)}\| \psi_{T}w\| _{(1,s,\frac12
-\theta)}\| \psi_{T}u\| _{(1,s,\frac12)}\\
& \hspace{0.4cm} +c\left(  \| \psi_{T}u\| _{(2,\frac12,0)}^{2}+\| \psi_{T}w\| _{(2,\frac12
,0)}^{2}\right)  \| \psi_{T}u\| _{(1,s,0)}\\
& \hspace{0.4cm} +c\| \psi_{T}u\| _{(2,\frac12,0)}\| \psi_{T}w\| _{(2,\frac12
,0)}\| \psi_{T}u\| _{(1,s,0)}\\
&  \leq c\| \overrightarrow{{u}_{0}}\| _{H^{s}\times H^{s}
}+cT^{\theta-\theta^{\prime}+\varepsilon}\|
\overrightarrow{u}\| _{Y_{s}\times Y_{s}}^{3}\leq
c\| \overrightarrow{{u}_{0}}\|
_{H^{s}\times H^{s}}+cT^{\theta-\theta^{\prime}+\varepsilon}a^{3}\text{,}%
\end{align*}
for $\overrightarrow{u}\in$ $Y_{T}^{a}$ and $0<\theta^{'}
-\varepsilon<\theta<\frac{1}{12}$. Analogously, for
$\overrightarrow
{u},\overrightarrow{v}\in$ $Y_{T}^{a}$ we have that%
\begin{equation}\label{4.47}
\|
\Phi(\overrightarrow{u})-\Phi(\overrightarrow{v})\|
_{Y_{s}\times Y_{s}}\leq
cT^{\theta-\theta^{'}+\varepsilon}a^{2}\|
\overrightarrow{u}-\overrightarrow{v}\| _{Y_{s}\times
Y_{s}}\text{.}
\end{equation}

We choose $\dfrac{a}{2}=c\| \overrightarrow{u_0}\| _{H^{s}}$ and $T$
small enough, such that $cT^{^{\theta-\theta^{'}+\varepsilon}}%
a^{3}\leq\dfrac{a}{2}$, then $\|
\Phi(\overrightarrow{u})\| _{Y_{s}\times Y_{s} }\leq a$
and $\Phi$ is a contraction. Uniqueness and continuous dependence follow in standard way.

\section{Proof of Theorem \ref{teo1.8}}

First we obtain conserved quantities for (\ref{1.1}) with $\sigma_{\alpha}=\sigma_{\beta}=\sigma_{\mu}=1$.

For $\Omega=\mathbb{T}$, define
\[
H_{0}(u,w)=2\mu\operatorname{Im}\left(  \int_{\Omega}((u\overline{u}_{x}%
)^{2}+(w\overline{w}_{x})^{2})dx\right)
+4\mu\operatorname{Im}\left(
\int_{\Omega}(u\overline{u}_{x}w\overline{w}_{x})dx\right)  \text{,}%
\]
\begin{align*}
H_{1}(u,w)  &  =\left(  \beta+2\mu\right)  \int_{\Omega}\left[ |
u_{x}| ^{2}\partial_{x}(| u| ^{2})+| w_{x}| ^{2}\partial_{x}\left(
| w| ^{2}\right)
\right]  dx\\
& \hspace{0.4cm} +\left(  \beta+2\mu\right)  \int_{\Omega}\left[  | u_{x}%
| ^{2}\partial_{x}(| w| ^{2})\right]  dx\\
& \hspace{0.4cm} +\left(  \beta+2\mu\right)  \int\left[  | w_{x}|
^{2}\partial_{x}\left(  | u| ^{2}\right)  \right]
dx+4\alpha\operatorname{Im}\left(  \int_{\Omega}(u\overline{u}_{x}%
w\overline{w}_{x})dx\right) \\
&\hspace{0.4cm}  +2\alpha\operatorname{Im}\left(  \int_{\Omega}((u\overline{u}_{x}%
)^{2}+(w\overline{w}_{x})^{2})dx\right)  \text{,}%
\end{align*}%
\begin{align*}
H_{2}(u,v)  &  =\left(  -\frac{1}{2}\beta+2\mu\right)
\int_{\Omega}\left(
\partial_{x}\left(  | u| ^{2}\right)  |
w| ^{4}+\partial_{x}\left(  | w| ^{2}\right)
| u| ^{4}\right)  dx\\
& \hspace{0.4cm} +q\operatorname{Im}\int_{\Omega}((u\overline{u}_{x})^{2}+(w\overline{w}%
_{x})^{2})dx\\
& \hspace{0.4cm} +\frac{3}{2}\gamma\int_{\Omega}\left[  | u_{x}|
^{2}\partial_{x}(| u| ^{2})+| w_{x}| ^{2}\partial_{x}\left( | w|
^{2}\right)  \right]  dx\text{
,}%
\end{align*}%
\begin{equation*}
\begin{split}
H_{3}(u,w)  &  =\left(  \frac{1}{2}\beta-2\mu\right)
\int_{\Omega}\left(
\partial_{x}\left(
| u| ^{2}\right)  | w| ^{4}+\partial_{x}\left(  | w| ^{2}\right) |
u| ^{4}\right)
dx\\
&+2q\operatorname{Im}\left(
\int_{\Omega}(u\overline{u}_{x}w\overline{w}_{x})dx\right) \\
& \hspace{0.4cm} +\frac{3}{2}\gamma\int_{\Omega}\left[  | u_{x}|
^{2}\partial_{x}(| w| ^{2})+| w_{x}| ^{2}\partial_{x}\left( | u|
^{2}\right)  \right] dx\text{.}
\end{split}
\end{equation*}

\begin{lemma}
Let $\overrightarrow{u_0}=(u_{0},w_{0})\in H^{s'}(\Omega
\mathbb{)}\times H^{s'}(\Omega\mathbb{)}$ with $s'$
large enough and $\overrightarrow{u}\in
C([-T,T];H^{s'}(\Omega\mathbb{)\times
}H^{s'}(\Omega\mathbb{))}$ solutions of (1.1) with
$\sigma_{\alpha}=\sigma_{\beta}=\sigma_{\mu}=1$. Then
\begin{equation}\label{5.1}
i\partial_{t}\left(  \int_{\Omega}(u\overline{u}_{x}+w\overline{w}%
_{x})dx\right)  =H_{0}(u,w)\text{,}
\end{equation}
\begin{equation}\label{5.2}
\partial_{t}\left(  \| u_{x}\| _{L^{2}\left(  \Omega\right)
}^{2}+\| w_{x}\| _{L^{2}\left(  \Omega\right)
}^{2}\right)
=-H_{1}(u,w)\text{,}
\end{equation}%
\begin{equation}\label{5.3}
\frac{1}{2}\partial_{t}\left(  \| u_{x}\|
_{L^{4}\left( \Omega\right)  }^{4}+\| w_{x}\|
_{L^{4}\left(  \Omega\right)
}^{4}\right)  =-H_{2}(u,w)\text{,}
\end{equation}%
\begin{equation}\label{5.4}
\partial_{t}\int_{\Omega}| u| ^{2}| w|
^{2}dx=-H_{3}(u,w)\text{.}
\end{equation}

\end{lemma}

\begin{proof}
Multiply the first equation in (\ref{1.1}) by $\overline{u}_{x}$,
second equation by $\overline{w}_{x}$, integrate in $x$ and take
the real part to obtain
\[
\operatorname{Re}\left(
2i\int_{\Omega}u_{t}\overline{u}_{x}dx\right)
=i\int_{\Omega}(u_{t}\overline{u}_{x}-\overline{u}_{t}u_{x})dx=i\partial
_{t}\left(  \int_{\Omega}u\overline{u}_{x}dx\right)  \text{,}%
\]%
\begin{align*}
\operatorname{Re}\left(  2i\mu\int_{\Omega}u\partial_{x}(| u|
^{2})\overline{u}_{x}dx\right)   & =-\operatorname{Im}\left(
2\mu\int_{\Omega}u\partial_{x}(| u| ^{2})\overline{u}%
_{x}dx\right) \\
&  =-2\mu\operatorname{Im}\left(  \int_{\Omega}\left(  u\overline{u}%
_{x}\right)  ^{2}dx\right)  \text{,}%
\end{align*}%
\begin{align*}
&  \operatorname{Re}\left( 2i\mu\int_{\Omega}u\partial_{x}(| w|
^{2})\overline{u}_{x}dx\right)  +\operatorname{Re}\left(
2i\mu\int_{\Omega}w\partial_{x}(| u| ^{2})\overline{w}%
_{x}dx\right) \\
&  =-4\mu\operatorname{Im}\left(
\int_{\Omega}(u\overline{u}_{x}w\overline
{w}_{x})dx\right)  \text{,}
\end{align*}
\begin{align*}
&  \operatorname{Re}\left(  2\alpha\int_{\Omega}u(| u|
^{2}+| w| ^{2})\overline{u}_{x}\right)  +\operatorname{Re}%
\left(  2\alpha\int_{\Omega}w(| u| ^{2}+|
w| ^{2})\overline{w}_{x}\right) \\
&   =-\alpha\int_{\Omega}\partial_{x}(| w| ^{2})| u| ^{2}dx\text{
}+\alpha\int_{\Omega}\partial_{x}(|
w| ^{2})| u| ^{2}dx=0\text{ .}%
\end{align*}
The equality (\ref{5.1}) is obtained by adding the previous resulting equations.

To obtain (\ref{5.2}) multiply the first equation in (\ref{1.1})
by $\overline{u}_{xx}$, the second equation by $\overline
{w}_{xx}$, integrate in $x$, take the imaginary part and add the
resulting equations. The equality (\ref{5.3}) is obtained
multiplying the first equation in (\ref{1.1}) by $| u|
^{2}\overline{u}$, the second equation by $| w| ^{2}\overline{w}$,
integrate in $x$ , take the imaginary part and add the resulting
equations. Similarly, we obtain (\ref{5.4}) multiplying the first
equation of (\ref{1.1}) by $| w| ^{2}\overline{u}$, the second
equation by $| u| ^{2}\overline{w}$, integrate in $x$, take the
imaginary part and sum the resulting equations.

\end{proof}

\begin{lemma}
Let $\overrightarrow{u_0}\in H^{1}(\Omega\mathbb{)}\times
H^{1}\left( \Omega\right)  $ and $\overrightarrow{u}\in C\left(
[-T,T];H^{1}(\Omega \mathbb{)\times}H^{1}(\Omega\mathbb{)}\right)  $
solution of (\ref{1.1}).
Then,
\begin{equation}\label{5.5}
I_{1}(u,w)=\| u\| _{L^{2}\left(  \Omega\right)  }^{2}+\| w\| _{L^{2}\left(  \Omega\right)  }^{2}
=I_{1}(u_{0},w_{0}),
\end{equation}
\begin{equation}\label{5.6}
\begin{split}
I_{2}(u,v) =&\,i\left( -3\gamma\alpha\!+\!\beta q\!+\!2\mu q\right)
\int_{\Omega}(u\overline{u}_{x}+w\overline{w}_{x})dx
+\frac{3}{2}\gamma\int_{\Omega}(|u_{x}|^2+|w_{x}|^2)\,dx \\
&  +\frac{1}{2}\left(  \beta+2\mu\right) \int_{\Omega}(|u|^4+|w|^4)\,dx
+\left(\beta+2\mu\right) \int _{\Omega}|u|^2| w|^2\,dx\\
 =&\, I_{2}(u_0,w_0).
\end{split}
\end{equation}

\end{lemma}

\begin{proof}
The following combination of
(\ref{5.1})-(\ref{5.4}) lead to (\ref{5.6})
\[
\left(  \frac{-3\gamma\alpha+(\beta+2\mu)q}{2\mu}\right)
 (4.1)+\frac{3\gamma }{2}(4.2)+\left(  \beta+2\mu\right)
(4.3)+\left( \beta+2\mu\right)
(4.4)\text{.}%
\]
\end{proof}
The global solution in Theorem \ref{teo1.8} follows from Theorem
1.6 and (4.5)-(4.6). \\

{\center{\bf{Final Remark}}}
\\

An interesting mathematical problem would be to study the local
well-posedness on the torus of the related higher-order nonlinear
Schr\"odinger system
\begin{equation*}
\left\{
\begin{array}
[c]{l}%
2i\partial_{t}u+q_1\partial_{x}^{2}u+i\gamma_1\partial_{x}^{3}u=F_{1}(u,w)\\
2i\partial_{t}w+q_2\partial_{x}^{2}w+i\gamma_2\partial_{x}^{3}w=F_{2}(u,w)\text{,}\\
u(x,0)=u_0,   w(x,0)=w_0,
\end{array}
\right.
\end{equation*}
when $q_1 \neq q_2$ and $\gamma_1\neq  \gamma_2.$\\
Note that in this case the resonant regions would be different
from the ones in \cite{4}, so the proofs of \cite{4} would need to
be carefully modified.

%\end{comments}

%%%%%%%%%%%%%%%%%%%%%%%%%%%%%%%%%%%%%%%%%%%%%%%%%%

\end{document}